\documentclass{amsart}
\usepackage{amscd,amssymb}

\newcounter{noindnum}[subsection]
\renewcommand{\thenoindnum}{\roman{noindnum}}
\newcommand{\noindstep}{\refstepcounter{noindnum}{{\rm(}\thenoindnum}\/{\rm)} }
\newcommand{\stepzero}{\setcounter{noindnum}{0}}
\renewcommand{\phi}{\varphi}
\renewcommand{\epsilon}{\varepsilon}
\renewcommand{\emptyset}{\varnothing}

\newcommand{\bE}{\mathbf E}
\newcommand{\bG}{\mathbf G}
\newcommand{\bH}{\mathbf H}

\newcommand{\bP}{\mathbf P}

\newcommand{\bU}{\mathbf U}

\newcommand{\cA}{\mathcal A}

\newcommand{\cE}{\mathcal E}
\newcommand{\cF}{\mathcal F}
\newcommand{\cG}{\mathcal G}
\newcommand{\cH}{\mathcal H}

\newcommand{\cO}{\mathcal O}
\newcommand{\cP}{\mathcal P}
\newcommand{\cT}{\mathcal T}

\newcommand{\cX}{\mathcal X}

\newcommand{\C}{\mathbb C}
\newcommand{\R}{\mathbb R}
\renewcommand{\P}{\mathbb P}
\newcommand{\A}{\mathbb A}

\newcommand{\Aff}{\mathrm{Aff}}
\newcommand{\Et}{\mathrm{\acute Et}}
\DeclareMathOperator{\Id}{Id}

\DeclareMathOperator{\Inn}{\mathbf{Inn}}
\DeclareMathOperator{\Out}{\mathbf{Out}}
\DeclareMathOperator{\Aut}{\mathbf{Aut}}
\DeclareMathOperator{\Orb}{Orb}
\DeclareMathOperator{\Iso}{\mathbf{Iso}}
\DeclareMathOperator{\ISO}{Iso}
\DeclareMathOperator{\spec}{Spec}
\DeclareMathOperator{\Hom}{Hom}
\newcommand{\Gm}{\mathop{\mathbb G_m}}

\newcommand{\fk}{\mathfrak k}

\newcommand{\fm}{\mathfrak m}

\theoremstyle{plain}
\newtheorem{theorem}{Theorem}

\newtheorem{proposition}{Proposition}[section]
\newtheorem{lemma}[proposition]{Lemma}
\newtheorem*{corollary*}{Corollary}
\newtheorem{corollary}[proposition]{Corollary}

\theoremstyle{definition}
\newtheorem{definition}[proposition]{Definition}

\theoremstyle{remark}
\newtheorem{remark}[proposition]{Remark}
\newtheorem{remarks}[proposition]{Remarks}
\newtheorem*{remark*}{Remark}
\newtheorem{example}[proposition]{Example}

\newtheorem{question}{Question}

\begin{document}

\title[Principal bundles over $\A^1$]{Affine Grassmannians of group schemes and exotic principal bundles over $\A^1$}


\thanks{This article appeared in the American Journal of Mathematics, Volume 138, Issue 4, 2016, pages 879-906, Copyright \copyright 2016, Johns Hopkins University Press.}

\begin{abstract}
Let $\bG$ be a simple simply-connected group scheme over a regular local scheme $U$. Let $\cE$ be a principal $\bG$-bundle over $\A^1_U$ trivial away from a subscheme finite over $U$. We show that $\cE$ is not necessarily trivial and give some criteria of triviality. To this end, we define affine Grassmannians for group schemes and study their Bruhat decompositions for semi-simple group schemes. We also give examples of principal $\bG$-bundles over $\A^1_U$ with split $\bG$ such that the bundles are not isomorphic to pullbacks from $U$.
\end{abstract}

\author{Roman Fedorov}
\email{rmfedorov@gmail.com}
\address{Mathematics Department, 138 Cardwell Hall, Kansas State University, Manhattan, KS 66506, USA}

\maketitle

\section{Introduction}
In 1976 Daniel Quillen and Andrei Suslin independently proved a conjecture of Serre that an algebraic vector bundle over an affine space is algebraically trivial (see~\cite{QuillenOnSerre,SuslinOnSerre}). A few years earlier, Hyman Bass (see~\cite[Problem~IX]{BassKTheory}) asked a more general question (see also~\cite{QuillenOnSerre}): Let $R$ be a regular ring, is every vector bundle over $\A^1_R:=\spec R[t]$ isomorphic to the pullback of a vector bundle over $\spec R$? This is now known as Bass--Quillen problem. Note that it is enough to consider the case when $R$ is a regular local ring (see~\cite[Thm.~1]{QuillenOnSerre}). The problem was solved by Hartmut Lindel in~\cite{LindelOnBassQuillen} in the geometric case, that is, when $R$ is a localization of a $k$-algebra of finite type, where $k$ is a field.

We can ask a more general question: consider a regular local $k$-algebra $R$ and let $\bG$ be a simple simply-connected group scheme over $R$.
\begin{question}\label{question2}
Let $\cE$ be a principal $\bG$-bundle over the affine line $\A^1_R$. Is $\cE$ isomorphic to the pullback of a principal $\bG$-bundle over $\spec R$?
\end{question}

Using the standard relation between principal $SL(n,R)$-bundles and rank $n$ vector bundles, we see that the above question reduces to the Bass--Quillen problem (=Lindel's Theorem if $R$ is of essentially finite type), when $\bG$ is the special linear group.

Note that the answer to Question~\ref{question2} is positive if $R$ is a perfect field by a theorem of Raghunathan and Ramanathan (see~\cite{RagunathanRamanthan} and~\cite{GilleTorseurs}). We will see that the answer is in general negative even if we assume that $\bG$ is a split group, and $k$ is the field of complex numbers (see Theorem~\ref{th:exot} and Example~\ref{example}, where $\bG=Spin(7,\C)$). To the best of my knowledge such examples were not known before (while examples with algebraically non-closed $k$, e.g.~$k$ being the field of real numbers, were known before, see Remark~\ref{rem:CounterExamples}\eqref{rem:Parimala}).

The principal bundles over $\A^1_R$ we construct have the following property: they are isomorphic to the pullback of principal bundles over $\spec R$ on the complement of a subscheme in $\A^1_R$ finite over $\spec R$. In this case, replacing the group scheme $\bG$ by a strongly inner form and `twisting' the bundle we may assume that the bundle is trivial away from the subscheme (see the proof of Theorem~\ref{th:exot} for details). Thus we arrive at the following question.

\begin{question}\label{question}
Let $\cE$ be a principal $\bG$-bundle over the affine line $\A^1_R$. Assume that~$\cE$ is trivial on the complement of a subscheme finite over~$\spec R$. Does this imply that $\cE$ is trivial?
\end{question}
A positive answer to this question was obtained in~\cite[Thm.~1.3]{PaninStavrovaVavilov} in the case, when the group scheme $\bG$ is isotropic. Using the technique of nice triples (similar to standard triples of Voevodsky), Panin, Stavrova, and Vavilov derived from this statement \emph{the conjecture of Grothendieck and Serre on principal bundles\/} for isotropic group schemes (see Section~\ref{subsect:GrSerre} for more details).

There was a certain hope that Question~\ref{question} would have a positive answer without the isotropy condition, implying that the conjecture of Grothendieck and Serre holds without the isotropy condition. However, using the technique of affine Grassmannians, we will show that the answer is in general \emph{negative}. A counterexample is given in this paper (see Theorem~\ref{th:exot} and Example~\ref{example}). We give some criteria for principal $\bG$-bundles as in Question~\ref{question} to be trivial, see Theorem~\ref{Th:A} below.

Note that a statement slightly weaker, than the positive answer to Question~\ref{question}, was proved by Panin and the author, see~\cite[Thm.~3]{FedorovPanin}. This statement was still sufficient to prove the conjecture of Grothendieck and Serre for regular local rings containing infinite fields. The idea is that, instead of studying the triviality of bundles over $\A^1_R$ one should study the triviality over $\P^1_R-Y$, where $Y$ is finite and \'etale over $\spec R$. This is also discussed in Section~\ref{sect:generalization} below (cf.~Theorem~\ref{Th:B}).

\subsection{Affine Grassmannians}
The proofs of the announced results are based on the technique of affine Grassmannians. Let $T=\spec A$ be an affine scheme. We define ``the formal disc'' $D_T$ over $T$ as $\spec A[[t]]$, where $A[[t]]$ is the ring of formal power series with coefficients in $A$. Similarly, let $A((t))=A[[t]][t^{-1}]$ denote the ring of formal Laurent series. Let $\dot D_T:=\spec A((t))$ be the ``punctured formal disc over $T$''.

Next, let $U$ be any connected affine scheme. Let $\Aff/U$ be the (big) \'etale site of affine schemes over $U$. Recall that a $U$-space is a sheaf on $\Aff/U$.

Let $\bG$ be a smooth affine $U$-group scheme. In Section~\ref{sect:Grassm} we define the affine Grassmannian of $\bG$ as the sheafification of the presheaf $T\mapsto\bG(\dot D_T)/\bG(D_T)$ (here~$T$ is an affine $U$-scheme). We show that this affine Grassmannian parameterizes principal $\bG$-bundles over $\P^1_U$ with trivialization on $\A^1_U$.

Then we develop a Bruhat decomposition of affine Grassmannians in the case when $\bG$ is semi-simple and $U$ is a scheme over a field. While these results are not very surprising, to the best of our knowledge this was not done before for non-split group schemes.

\subsection{Acknowledgments} The author obtained main results underlying this paper, while working with I.~Panin on~\cite{FedorovPanin}. The idea of the current paper emerged in a conversation with P.~Gille and J.-L. Colliot--Th\'el\`ene, while the author visited \'Ecole Normale Sup\'erieure. The author is also thankful to A.~Stavrova and B.~Antieau for their interest in this work. The author is also grateful to D.~Arinkin, J.~Humphreys, C.~Sorger, and K.~Zaynullin for clarifying certain points. The author talked about these results at the conference on Torsors, Nonassociative Algebras and Cohomological Invariants at Fields Institute. He wants to thank both the Institute and the organizers.

While working on a revision of this paper, the author was a member of Max Planck Institute for Mathematics in Bonn, he wants to thank the institute for the perfect working atmosphere. Also, the author was partially supported by the NSF grant DMS-1406532. Finally, he wants to thank the anonymous referee for useful comments.

\section{Main results}
\subsection{Conventions}
Let $U$ be a scheme and $\bG$ be a group scheme over~$U$. We always assume that $\bG$ is affine, flat, and of finite presentation over $U$. Recall that a $U$-scheme $\cG$ with a left action of $\bG$ is called \emph{a principal $\bG$-bundle over $U$}, if $\cG$ is faithfully flat and quasi-compact over $U$ and the action is simply transitive, that is, the obvious morphism $\bG\times_U\cG\to\cG\times_U\cG$ is an isomorphism, see~\cite[Sect.~6]{FGA1}. If~$\bG$ is smooth over $U$, then it is well known that every principal $\bG$-bundle is trivial locally in \'etale topology (but in general not in Zariski topology).

Note that we can similarly define right $\bG$-bundles as schemes with right action of $\bG$ satisfying the same condition. In fact, every left $\bG$-bundle can be viewed as a right $\bG$-bundle by composing the action with the group inversion, and vice versa. If $T$ is a $U$-scheme, we will use the term ``principal $\bG$-bundle over $T$'' to mean a principal $\bG\times_UT$-bundle over $T$. We often skip the adjective ``principal''. Finally, we often consider the following situation: $\phi:T'\to T$ is a morphism, $\cE$ is a $\bG$-bundle over $\P_T^1$. We denote $(\phi\times\Id_{\P^1})^*\cE$ by $\phi^*\cE$ to simplify notation.

By a $k$-group, where $k$ is a field, we mean a smooth $k$-group scheme (of finite type over $k$).

We define simple, semi-simple, and simply-connected group schemes as in~\cite{SGA3-3} (see~[Exp.~XIX, Def.~2.7], [Exp.~XXI, Def.~6.2.6], and [Exp.~XXIV, Sect 5.3] \emph{ibid}.). In particular, a $T$-group scheme $\bG$ is \emph{semi-simple}, if it is affine, of finite type, smooth over $T$, and the geometric fibers are connected semi-simple groups.

Recall that a simple $T$-group scheme is called \emph{isotropic}, if it contains a torus isomorphic to~$\mathbb G_{m,T}$. Note that if $T$ is semi-local, then by~\cite[Exp.~XXVI, Cor.~6.14]{SGA3-3} this is equivalent to the condition that the restriction of $\bG$ to each connected component of $T$ contains a proper parabolic subgroup scheme. A group scheme that is not isotropic is called \emph{anisotropic}.

\subsection{Principal bundles over affine lines}\label{sect:PrincipalOverLines}
Let $R$ be a regular local $k$-algebra, where $k$ is an infinite field. Set $U=\spec R$, let $u\in U$ be the closed point. Let~$\bG$ be a simple simply-connected group scheme over $U$. Let $\bG_u$ be the fiber of $\bG$ over~$u$ so that $\bG_u$ is a simple $k$-group. Let $\cE$ be a principal $\bG$-bundle over $\A^1_U$ trivial on a~complement of a finite over $U$ subscheme $Z$. We ask whether $\cE$ is trivial. Amazingly, we can give a conjecturally complete answer to this question. Note that $Z$ is closed in $\P_U^1$; choose a trivialization of $\cE$ on $\A_U^1-Z$. We can extend~$\cE$ to $\P_U^1$ by gluing $\cE$ with the trivial $\bG$-bundle over $\P_U^1-Z$; denote the obtained bundle by $\hat\cE$.

\begin{theorem}\label{Th:A} Assume that $\cE$ is a principal $\bG$-bundle over $\A^1_U$ whose restriction to $\A^1_U-Z$ is trivial. Then

\stepzero\noindstep\label{thpr:LineA} If $\hat\cE_u:=\hat\cE|_{\P_u^1}$ is a trivial $\bG_u$-bundle, then $\hat\cE$ is trivial. (Thus, $\cE$ is also trivial.)

\noindstep\label{thpr:LineB} If $\bG$ is isotropic, then $\cE$ is trivial.

\noindstep\label{thpr:LineC} If $\bG$ is anisotropic at the generic point of $U$, and $\hat\cE_u$ is not a trivial $\bG_u$-bundle, then $\cE$ is not trivial.
\end{theorem}
\begin{proof}
The theorem follows from more general Theorem~\ref{Th:B} below.
\end{proof}
\begin{remarks}
\stepzero\noindstep The reader may ask what happens if $\bG$ is isotropic at the generic point of $U$ but anisotropic on $U$. This would contradict a conjecture, predicting that a principal bundle under a reductive group can be reduced to a parabolic subgroup over a regular local ring, if it can be reduced to this subgroup generically (cf.~the Grothendieck--Serre conjecture, Section~\ref{subsect:GrSerre}). To the best of my knowledge this conjecture belongs to Colliot--Th\'el\`ene. It follows from~\cite{OjangurenPanin2,PaninPetrovPurity}, that the conjecture holds for many simple groups schemes.

\noindstep Note that part~\eqref{thpr:LineA} of the theorem follows from the fact that the trivial bundle is open in the moduli of $\bG_u$-bundles over $\P_u^1$ (see~\cite[Prop.~5.1]{FedorovPanin}). Part~\eqref{thpr:LineB} follows from~\cite[Thm.~1.3]{PaninStavrovaVavilov}. Part~\eqref{thpr:LineC} is one of the main results of this paper; it is proved below using the technique of affine Grassmannians.

\noindstep If $\bG_u$ is anisotropic, then the bundle $\hat\cE_u$ is trivial by~\cite[Thm.~3.8]{GilleTorseurs}. Thus in this case only possibility~\eqref{thpr:LineA} can occur in the theorem.
\end{remarks}

The theorem has the following unexpected corollary.
\begin{corollary}
In the situation of the theorem assume that $\bG$ is anisotropic at the generic point of $U$. Then the following are equivalent

\stepzero\noindstep $\cE$ is trivial over $\A^1_U$;

\noindstep $\hat\cE$ is trivial over $\P^1_U$;

\noindstep $\hat\cE_u$ is trivial over $\P^1_u$.
\end{corollary}

Another question is whether the situation of Theorem~\ref{Th:A}\eqref{thpr:LineC} is possible at all. The answer is given by the following extension theorem.

\begin{theorem}\label{ThMaina}
Let $U$ be any affine scheme; let $u\in U$ be a closed point; let $\bG$ and $\bG_u$ be as before. Assume that $Z\subset\A_U^1$ is finite and \'etale over $U$, the restriction of the group scheme $\bG_Z:=\bG\times_UZ$ to each connected component of $Z$ has a proper parabolic subgroup scheme, and the fiber $Z_u$ has a $k(u)$-rational point, where $k(u)$ is the residue field of $u$. Let $E$ be any $\bG_u$-bundle over $\P^1_u$ trivial at the generic point. Then there is a $\bG$-bundle $\cE$ over $\P^1_U$ such that $\cE|_{\P^1_u}\approx E$ and $\cE$ is trivial away from~$Z$.
\end{theorem}

This is the second main result of the paper. One can give a proof very similar to that of~\cite[Thm.~3]{FedorovPanin}. We will give a proof using affine Grassmannians in Section~\ref{ProofThMaina}.

\begin{corollary}\label{cor:example}
Assume that $U$ is the spectrum of a regular local $k$-algebra, where~$k$ is an infinite field. Let $\bG$ be a simple simply-connected $U$-group scheme such that~$\bG$ is anisotropic at the generic point of $U$ but isotropic at the closed point $u$ of $U$. Then there exists a non-trivial $\bG$-bundle $\cE$ over $\A^1_U$ such that $\cE$ is trivial on $\A^1_U-Z$ for a certain $Z$ finite and \'etale over $U$.
\end{corollary}
\begin{proof}
Recall that generically trivial $\bG_u$-bundles over $\P^1_u$ are classified by the Weyl group orbits of the co-characters of a maximal split torus (see~\cite[Thm.~3.8(b)]{GilleTorseurs} and~\cite{GilleErratum}). Thus, since $\bG_u$ is isotropic, there is a non-trivial $\bG_u$-bundle $E$ over $\P^1_u$ such that~$E$ is trivial generically. By~\cite[Prop.~4.1]{FedorovPanin} we can choose $Z\subset\A^1_U$ satisfying the conditions of Theorem~\ref{ThMaina}. It remains to apply Theorem~\ref{ThMaina} and Theorem~\ref{Th:A}\eqref{thpr:LineC}.
\end{proof}
This corollary shows that~\cite[Thm.~1.3]{PaninStavrovaVavilov} is not true without the isotropy condition. Also, it is clear that a bundle  $\cE$ in this corollary is not isomorphic to the pullback of any $\bG$-bundle over $U$. Indeed, since $U$ is local and $k$ is infinite, we can choose $a\in k$ such that $a\times_k U$ does not intersect $Z$ (we view $a$ as a rational point of $\A_k^1$). Now, if $\cE$ is isomorphic to the pullback of $\cF$, then restricting $\cE$ to $a\times_k U$, we would see that $\cF$ is trivial.

\subsection{Counterexamples}\label{sect:counter}
Now we will give some explicit examples. Through the end of Section~\ref{sect:counter} we assume that $k$ is algebraically closed.
Let $G$ be a $k$-group, let $H\subset G$ be a $k$-subgroup. Recall that an $H$-bundle $\cH$ is \emph{a reduction of a $G$-bundle $\cG$ to $H$} if there is an isomorphism of $G$-bundles $G\times^H\cH\approx\cG$ (see Section~\ref{sect:associated} for more details on associated spaces). Recall also that the group scheme of automorphisms of a principal $G$-bundle is called \emph{a strongly inner form\/} of $G$ (again, see Section~\ref{sect:associated} for more details).

\begin{theorem}\label{th:exot}
Let $G$ be a simple simply-connected $k$-group, where $k$ is an algebraically closed field. Let $U$ be a regular local $k$-scheme. Assume that there is a~$G$-bundle over~$U$ that cannot be reduced to a proper parabolic subgroup of $G$ at the generic point of~$U$. Then

\stepzero\noindstep\label{th:exotA} There is a $G$-bundle over $\A^1_U$ not isomorphic to the pullback of a $G$-bundle over $U$.

\noindstep\label{th:exotB} There is a $U$-group scheme $\bG$ over $U$ such that $\bG$ is a strongly inner form of $G$, and a non-trivial $\bG$-bundle $\cE$ over $\A^1_U$ such that $\cE$ is trivial away from a subscheme $Y\subset\A^1_U$ finite and \'etale over $U$.
\end{theorem}

This theorem will be proved in Section~\ref{ProofThExot}.

\begin{example}\label{example}
An algebraic group $G$ is called \emph{special}, if every $G$-bundle is locally trivial in Zariski topology. Consider $G=Spin(7,k)$. Then $G$ is a simply-connected simple $k$-group of type $B_3$. According to the classification of Grothendieck~\cite[Thm.~3]{GrothendieckTorsion}, $G$ is not special, that is, there exists a smooth connected
$k$-variety $X$ and a principal $G$-bundle $\cF$ over $X$ such that $\cF$ is not Zariski locally trivial. In fact, as Corollary~3 of~\cite[Thm.~1]{GrothendieckTorsion} shows, we can take any faithful representation $Spin(7,k)\hookrightarrow SL(n,k)$ and put $X=SL(n,k)/Spin(7,k)$, $\cF=SL(n,k)$. Let $x\in X$, set $U:=\spec\cO_{X,x}$.

\begin{lemma}
The $G$-bundle $\cF$ cannot be reduced to a proper parabolic subgroup of~$G$ over $U$.
\end{lemma}
\begin{proof}
Let $P$ be a proper parabolic subgroup of $G$, let $L$ be the derived group of a~Levi factor of $P$. Then $L$ is trivial or semi-simple of type $B_2=C_2$, $A_2$, $A_1\coprod A_1$, or~$A_1$. By~\cite[Cor.~4.4]{Borel-Tits2} $L$ is also simply-connected, thus it is trivial or isomorphic to $Sp(4)$, $SL(3)$, $SL(2)\times SL(2)$, or $SL(2)$. If it was possible to reduce~$\cF$ to $P$ over~$U$, it would be possible to reduce it to $L$, but all possible groups $L$ are special, so every principal $L$-bundle over $U$ is trivial. Thus $\cF$ would be trivial over $U$. But then the positive answer to the conjecture of Grothendieck and Serre (see~Section~\ref{subsect:GrSerre} and~\cite[Thm.~1]{FedorovPanin}) would imply that $\cF$ is Zariski locally trivial over $X$.
\end{proof}

Since the point $x$ was arbitrary, we see that $\cF$ cannot be reduced to a proper parabolic subgroup of $G$ at the generic point $\omega$ of $X$. Thus Theorem~\ref{th:exot} is applicable to $Spin(7,k)$ and $U$. In particular, there is a $Spin(7,k)$-bundle over $\A^1_U$ not isomorphic to a pullback from $U$. Also, according to the theorem, there is a strongly inner form $\bG$ of $Spin(7,k)$ over~$U$ and a non-trivial $\bG$-bundle $\cE$ over $\A^1_U$ such that~$\cE$ is trivial on a complement of a subscheme $Y\subset\A^1_U$ such that $Y$ is finite and \'etale over $U$.
\end{example}

\begin{remarks}\label{rem:CounterExamples}
\stepzero\noindstep We could have taken $G=Spin(8,k)$ or a $k$-group of type $G_2$ as well. We expect that similar examples exist for any non-special simple $k$-group.

\noindstep\label{rem:Parimala} An example of a principal $SO(4,\R)$-bundle over $\A^1_R$ not isomorphic to a pullback from $\spec R$ has been constructed by Parimala, see~\cite[Thm.~2.1]{ParimalaFailure}. (Here~$R$ is a localization of $\R[x]$.) However, in this case the field is not algebraically closed and the group is not simply-connected. Another counterexample with the same $R$ and $G=SL(1,Q)$, where $Q$ is the real quaternion algebra has been suggested by the referee. In this case, the group is not a strongly inner form of a split group (and again the field is not algebraically closed).
\end{remarks}

\subsection{A generalization}\label{sect:generalization} As before, let $U=\spec R$ be the spectrum of a regular local $k$-algebra, where $k$ is an infinite field; let $u\in U$ be the closed point. Let $\bG$ be a simple simply-connected $U$-group scheme. Let $Z\subset\P^1_U$  be a subscheme finite over $U$, let $Y\subset\P^1_U$ be a subscheme finite and \'etale over $U$. Assume that $\cE$ is a $\bG$-bundle over $\P^1_U$ trivial over $\P^1_U-Z$. We wonder whether it is trivial over $\P^1_U-Y$. Assume additionally that $Y_u:=Y\times_Uu$ has a $k(u)$-rational point, where $k(u)$ is the residue field of $u$. Note that $Y$ may be disconnected in which case it has many generic points.
\begin{theorem}\label{Th:B} Let $\cE$ be a $\bG$-bundle over $\P_U^1$ trivial on $\P_U^1-Z$.

\stepzero
\noindstep\label{thprB:LineA} If $\cE_u:=\cE|_{\P_u^1}$ is a trivial $\bG_u$-bundle, then $\cE$ (and thus $\cE|_{\P^1_U-Y}$) is trivial.

\noindstep\label{thprB:LineB} If $\bG_Y:=\bG\times_UY$ is isotropic, then $\cE|_{\P^1_U-Y}$ is trivial.

\noindstep\label{thprB:LineC} If $\bG_Y$ is anisotropic at every generic point of $Y$, and $\cE_u$ is not a trivial $\bG_u$-bundle, then $\cE|_{\P_U^1-Y}$ is not trivial.
\end{theorem}
Clearly, this theorem implies Theorem~\ref{Th:A} upon taking $Y$ to be the infinity divisor $\infty\times U$ of $\P_U^1$ (so that $\A_U^1=\P_U^1-Y$).
\begin{proof}
Part~\eqref{thprB:LineA} follows from~\cite[Prop.~5.1]{FedorovPanin} (or from~\cite[Thm.~9.6]{PaninStavrovaVavilov}). Part~\eqref{thprB:LineB} is essentially~\cite[Thm.~3]{FedorovPanin}. (In~\cite{FedorovPanin} we work with semi-local rings. In the present paper we only consider local rings to simplify notation. The generalization to semi-local case is, however, straightforward.) We will give a proof based on affine Grassmannians in Section~\ref{sect:modifications}. Part~\eqref{thprB:LineC} follows from the next, more general theorem.
\end{proof}

\begin{theorem}\label{ThMainb}
Let $U$ be any affine integral $k$-scheme and let $Y\subset\P_U^1$ be finite and \'etale over~$U$. Assume that $\bG$ is a semi-simple $U$-group scheme such that $\bG_Y$ is anisotropic at every generic point of $Y$. Assume that $\cE$ is a principal $\bG$-bundle over $\P^1_U$ such that  $\cE|_{\P^1_U-Y}$ is a trivial bundle. Then $\cE$ is a trivial bundle.
\end{theorem}
This theorem will be proved in Section~\ref{sect:exoticism}. The idea of the proof is as follows. The pair $(\cE,\tau)$, where $\tau$ is a trivialization of $\cE$ on $\P^1_U-Y$, can be viewed as a modification of the trivial bundle at $Y$, or equivalently as a $Y$-point of an affine Grassmannian. However, to any non-trivial modification one can assign a proper parabolic subgroup of $\bG$ at the generic point of $Y$. Since such subgroup does not exist, the modification is necessarily trivial. In other words, the bundle $\cE|_{\P^1_U-Y}$ admits a unique extension to $\P_U^1$. Note that similar considerations appeared in~\cite{RaghunathanOnAffineSpace}.

\begin{remark}
We see that Theorem~\ref{Th:B}\eqref{thprB:LineC} holds for all semi-simple group schemes. I expect that whole Theorem~\ref{Th:B} is true for all semi-simple group schemes. In Theorem~\ref{ThMaina}, however, if $\bG$ is not simply-connected, one should require that $E$ belongs to the connected component of the trivial bundle in the stack of $\bG_u$-bundles over~$\P_u^1$.
\end{remark}

\subsection{Affine Grassmannians} We briefly summarize the results of Section~\ref{sect:Grassm}.
Let~$U$ be a connected affine scheme, let $\bG$ be a smooth affine $U$-group scheme. In this case we define the \emph{affine Grassmannian\/} $Gr_\bG$. This is a sheaf on the big \'etale site $\Aff/U$ assigning to a $U$-scheme $T$ the set of isomorphism classes of pairs
\[
    \{(\cE,\tau)\,|\;\cE\text{ is $\bG$-bundle over }\P_U^1, \tau\text{ is a trivialization of $\cE$ on $\A^1_U$}\}.
\]
In the case, when $U$ is a $k$-scheme and $\bG$ is a semi-simple $U$-group scheme, we show that $Gr_\bG$ is a (strict) ind-projective ind-scheme over $U$ (Proposition~\ref{pr:indscheme}). We also develop a stratification
\[
    Gr_\bG=\coprod Gr^{\hat\lambda}_\bG,
\]
where $\hat\lambda$ ranges over the set of dominant co-characters of the split $k$-group of the same type as $\bG$ modulo the action of outer automorphisms (see Proposition~\ref{pr:twistedcells}). Finally, for each cell $Gr^{\hat\lambda}_\bG$ we construct a morphism to the scheme of parabolic subgroup schemes of $\bG$, see Proposition~\ref{Pr:TwistedBB}.

\subsection{Relation to a conjecture of Grothendieck and Serre}\label{subsect:GrSerre}
Let, as usual, $R$ be a regular local ring, $U=\spec R$. Let $\bG$ be a reductive $U$-group scheme, let $\cG$ be a principal $\bG$-bundle over $U$. A conjecture of Grothendieck and Serre (see~\cite[Remarque, p.31]{SerreFibres}, \cite[Remarque 3, p.26-27]{GrothendieckTorsion}, and~\cite[Remarque~1.11.a]{GrothendieckBrauer2}) predicts that $\cG$ is trivial if it is trivial over $\spec K$, where $K$ is the fraction field of~$R$.

In~\cite{PaninPurity,PaninStavrovaVavilov} Panin, Stavrova, and Vavilov proved the conjecture in the case, when $R$ contains an infinite field, $\bG$ is isotropic (in fact, they proved a generalization of the conjecture). One of the main steps in their proof was Theorem~\ref{Th:A}\eqref{thpr:LineB} (see~\cite[Thm.~1.3]{PaninStavrovaVavilov}). This is exactly where the isotropy condition was used. There was a hope that it would be possible to prove this theorem without the isotropy condition, thus proving the conjecture of Grothendieck and Serre without the isotropy condition.

A couple years later Panin and myself applied the technique of affine Grassmannians to the problem. It was soon clear that
Theorem~\ref{Th:A}\eqref{thpr:LineB} is unlikely to be true without the isotropy assumption. Soon, we figured out how to modify the statement: Theorem~\ref{Th:A}\eqref{thpr:LineB} was replaced by Theorem~\ref{Th:B}\eqref{thprB:LineB}; the latter theorem was sufficient to prove the conjecture of Grothendieck and Serre for all regular local rings~$R$ containing infinite fields and all reductive group schemes $\bG$.

\subsection{Organization of the paper}
In Section~\ref{sect:PrincipalBundles} we discuss associated spaces of principal bundles and strongly inner forms of group schemes. In Section~\ref{sect:gluing} we explain how to glue a principal bundle defined away from a divisor with a bundle over ``completed product'' of this divisor with a formal disc. In Section~\ref{sect:Grassm} we introduce the main objects of our study, that is, affine Grassmannians. We explain that affine Grassmannians are ind-schemes and study a Bruhat decomposition of an affine Grassmannian into the union of quasi-projective $U$-schemes. These results are used in Section~\ref{sect:exoticism} to prove Theorem~\ref{ThMainb} and in Section~\ref{sect:constructing} to prove Theorem~\ref{ThMaina}. In the second half of Section~\ref{sect:constructing} we reprove Theorem~\ref{Th:B}\eqref{thprB:LineB} using the affine Grassmannians.

\section{Generalities on principal bundles}\label{sect:PrincipalBundles}
\subsection{Generalities on associated spaces}\label{sect:associated} Let $S$ be a scheme; recall that an \emph{$S$-space\/} is a sheaf on the big \'etale site of schemes over $S$ (denoted by $\Et/S$.).

Let $\bH$ be an $S$-group scheme; let $\cH$ be a \emph{right} $\bH$-bundle over $S$. Let $\cX$ be an $S$-space acted upon by~$\bH$ on the left. The \emph{associated $S$-space $\cH\times^{\bH}\cX$} is the sheafification of the following presheaf on $\Et/S$
\[
T\mapsto(\cH(T)\times\cX(T))/\sim,
\]
where $\sim$ is the equivalence relation $(fg,x)\sim(f,gx)$ for all $f\in\cH(T)$, $g\in\bH(T)$, $x\in\cX(T)$. If $\bH$ is smooth over $S$, then the $\bH$-bundle $\cH$ is \'etale locally trivial, and one sees that $\cH\times^{\bH}\cX\to S$ is \'etale locally over $S$ isomorphic to $\cX\to S$.

Equivalently, composing the action of $\bH$ on $\cH$ with the group inversion, we can view~$\cH$ as a left $\bH$-bundle, then $\cH\times_S\cX$ acquires an action of $\bH$, and we have $\cH\times^\bH\cX=(\cH\times_S\cX)/\bH$. We also call $\cH\times^\bH\cX$ \emph{the twist of $\cX$ by $\cH$}. Similarly, if $\cX$ is an $S$-space with a right action of $\bH$, and $\cH$ is a left $\bH$-bundle, we can form the associated space $\cX\times^{\bH}\cH$.

More generally, let $\bH$ and $\cX$ be as above, let $S'$ be an $S$-scheme and let $\cH$ be a $\bH$-bundle over $S'$ (that is, an $\bH\times_SS'$-bundle). In this case we use the notation
\begin{equation}\label{eq:notation}
    \cH\times^{\bH}\cX:=\cH\times^{\bH\times_SS'}(\cX\times_SS').
\end{equation}
Again, if $\bH$ is smooth over $S$, then $\cH\times^{\bH}\cX\to S'$ is \'etale locally over $S'$ isomorphic to $\cX\times_SS'\to S'$.

The following proposition seems to be well-known. However, I was unable to find a reference, so I give a proof for the sake of completeness.
\begin{proposition}\label{pr:scheme}
Let $k$ be a field. Assume that $H$ is an affine $k$-group of finite type, $\cH$ is an $H$-bundle over a $k$-scheme $S$, and $\cX$ is a quasi-projective $k$-scheme. Then $\cH\times^H\cX$ is a scheme.
\end{proposition}
\begin{proof}
We will reduce the statement to the results of~\cite{SerreFibres}. First of all, the statement is obvious if $\cX$ is an affine scheme, because $\cH$ is locally trivial in \'etale topology, and affine schemes can be glued in \'etale topology.

Next, we claim that $\cH$ is \emph{isotrivial}, that is, for all $s\in S$ there is a Zariski neighborhood $S'$ of $s$ and an \'etale \emph{finite\/} morphism $S''\to S'$ such that the pullback of $\cH$ to $S''$ is a trivial $H$-bundle. Indeed, let $H\hookrightarrow GL(n)$ be a faithful representation (existing, e.g., by~\cite[Prop.~1.10]{BorelLinAlgGrps}), and consider the scheme $\cH\times^HGL(n)$ (this is a scheme because $GL(n)$ is an affine scheme). This is a principal $GL(n)$-bundle over $S$, and, replacing $S$ by a Zariski cover, we can assume that this bundle is trivial:
\[
    \cH\times^HGL(n)\approx S\times_k GL(n).
\]
(Indeed, using the equivalence between principal $GL(n)$-bundles and rank $n$ vector bundles, one checks that every $GL(n)$-bundle is Zariski locally trivial.) Now $\cH$ can be viewed as a reduction of the trivial $GL(n)$-bundle to $H$. As such, it corresponds to a morphism $S\to GL(n)/H$, in the sense that the bundle $\cH\to S$ is isomorphic to the pullback of $GL(n)\to GL(n)/H$, see~\cite[Proof of Lemma~2.2.3]{SorgerLecturesBundles}. Now the isotriviality of $\cH$ follows from~\cite[Prop.~3]{SerreFibres}, which claims that $GL(n)$ is an isotrivial $H$-bundle over $GL(n)/H$. It remains to use~\cite[Prop.~4]{SerreFibres}. (Note that there is a certain clash of terminology. The notion of `algebraic space' used in~\cite{SerreFibres} is different from what we usually mean by algebraic space now. In particular, algebraic spaces in~\cite{SerreFibres} are schemes.)
\end{proof}

Let $\bH'$ be a subgroup scheme of an $S$-group scheme $\bH$. Let $\cH$ be an $\bH$-bundle. A \emph{reduction of $\cH$ to $\bH'$} is an $\bH'$-bundle $\cH'$ together with an isomorphism of $\bH$-bundles $\cH'\times^{\bH'}\bH\simeq\cH$.

\subsection{Strongly inner forms and parabolic reductions}\label{sect:inner}
Now let $\bH$ be a reductive $S$-group scheme, $\cH$ and $\cH'$ be $\bH$-bundles over $S$. Then we have the $S$-scheme $\Iso_{\bH}(\cH,\cH')$ of $\bH$-bundle isomorphisms. The scheme $\Aut(\cH):=\Iso_{\bH}(\cH,\cH)$ is a reductive group scheme \'etale locally isomorphic to $\bH$ (our convention is that $\Aut(\cH)$ acts on the right on $\cH$). We call it \emph{a strongly inner form\/} of $\bH$. There is a natural action of $\Aut(\cH)$ on $\Iso_{\bH}(\cH,\cH')$ making the latter scheme an $\Aut(\cH)$-bundle over $S$.

\begin{proposition}\label{pr:parred}
Assume that $\bH$ is a split reductive $S$-group scheme, where $S$ is connected, $\cH$ is an $\bH$-bundle over $S$. If $\Aut(\cH)$ contains a proper parabolic subgroup scheme, then $\cH$ can be reduced to a proper parabolic subgroup scheme of~$\bH$.
\end{proposition}
\begin{proof}
Let $\bP\subset\Aut(\cH)$ be a proper parabolic subgroup scheme. Since $\bH$ is split, we can find a parabolic subgroup scheme $\bP'$ of $\bH$ such that $\bP$ and $\bP'$ are conjugate locally in \'etale topology. Consider the presheaf on $\Et/S$ defined by
\[
    T\mapsto\{s\in\cH(T):\bP'(T)s=s\,\bP(T)\}.
\]
Denote by $\cP$ the sheafification of this presheaf. We claim that $\cP$ is a principal $\bP'$-bundle over $S$. Indeed, it is clear that $\bP'$ acts on $\cP$, and we need to check that $\bP'$ is locally trivial in \'etale topology. Since the statement is \'etale local, we can assume that $\cH$ is a trivial bundle, then the statement follows from the fact that a parabolic subgroup is its own normalizer (see~\cite[Exp.~XXVI, Prop.~1.2]{SGA3-3}). It is easy to construct an $S$-morphism from the associated scheme $\bH\times^{\bP'}\cP$ to $\cH$. It remains to check that this morphism is an isomorphism. Again, it is an \'etale local statement, so we can assume that $\cH$ is trivial, making the statement obvious.
\end{proof}

\section{Gluing principal bundles}\label{sect:gluing}
In this section we assume that $U=\spec R$ is a connected affine scheme, $Y$ is a subscheme of $\A^1_U$ \'etale and finite over $U$. Note that $Y$ is automatically a closed subscheme of $\A^1_U$. Also, if $Y$ is non-empty, then it is surjective over $U$. Indeed, the projection $Y\to U$ is closed, since it is finite, and open, since it is \'etale.

For $Z=\spec A$, we define ``the formal disc over $Z$'' as $\spec A[[t]]$, where $A[[t]]$ is the ring of formal power series with coefficients in $A$. Denote this formal disc by~$D_Z$. Similarly, let $A((t))=A[[t]][t^{-1}]$ denote the ring of formal Laurent series. Let $\dot D_Z:=\spec A((t))$ be the ``punctured formal disc over $Z$''.

The main result of this section can be summarized as follows: \emph{There is a canonical commutative diagram of morphisms of $U$-schemes
\begin{equation}\label{eq:fpqccover}
\begin{CD}
\dot D_Y @>>> D_Y \\
@VVV @VVV \\
\P^1_U-Y @>>> \P^1_U.
\end{CD}
\end{equation}
Further, let $\bG$ be an affine flat $U$-group scheme. Given a principal $\bG$-bundle over $\P^1_U-Y$, a principal $\bG$-bundle over $D_Y$, and an isomorphism between their restrictions to $\dot D_Y$, we can glue the bundles into a principal $\bG$-bundle over $\P^1_U$.}

Below we will reduce this statement to the main result of~\cite{BeauvilleLaszlo}. The statement is well-known to specialists, at least in the case, when the projection $Y\to U$ is an isomorphism. Unfortunately, I could not find a reference. This gluing technique goes back to~\cite[Appendix]{FerrandRaynaud}. The advantage of~\cite{BeauvilleLaszlo} is that the schemes are not necessarily Noetherian.

\subsection{Constructing the commutative diagram} We need a general statement.
\begin{lemma}\label{lm:decomp}
There is a finite and \'etale surjective morphism $U'\to U$ such that $Y\times_UU'$ decomposes as $\coprod_{i=1}^l Y_i$ such that for all $i$ the projection $Y_i\to U'$ is an isomorphism.
\end{lemma}
\begin{proof}
Let $Y=\spec A$. Define
\[
    \deg(Y/U):=\max\dim_{R/\fm}(A/\fm A),
\]
where maximum is taken over all maximal ideals of $R$. This degree is finite, because it is bounded by the number of generators of $A$ as a module over $R$. In particular, $Y$ has finitely many connected components, since every component maps to $U$ surjectively. Thus we can assume that $Y$ is connected. We prove the lemma by induction on $\deg(Y/U)$. If $\deg(Y/U)=0$, then $Y=\emptyset$ and we are done. Otherwise, as was explained in the beginning of the section, the morphism $Y\to U$ is surjective.

Consider the projection to the second multiple $Y\times_U Y\to Y$. The diagonal morphism $\Delta:Y\to Y\times_U Y$ is \'etale
by~\cite[Ch.~6, Prop.~4.7(v)]{AltmanKleiman}. Thus $\Delta$ is open and closed so we can decompose $Y\times_U Y=\Delta(Y)\coprod Y'$. It is enough to prove the lemma for the morphism $Y'\to Y$ (instead of $Y\to U$). But
\[
    \deg(Y'/Y)=\deg(Y/U)-1,
\]
and we can apply the induction hypothesis.
\end{proof}

We define a morphism $j_Y$ as the composition
\[
    D_Y\to\A^1_Y\to\A^1_U.
\]
Here the first morphism is induced by the inclusion $A[t]\to A[[t]]$. The second morphism is the restriction of the group scheme multiplication morphism $\A^1_U\times_U\A^1_U\to\A^1_U$ to $\A^1_Y=\A^1_U\times_UY$. (Naively, $j_Y$ is thought as taking $(\epsilon,y)$ to $y+\epsilon$.) \emph{The right vertical arrow} in diagram~\eqref{eq:fpqccover} is the composition of $j_Y$ and the embedding $\A_U^1\to\P_U^1$.

We identify $Y$ with a closed subscheme of $D_Y$ given by the ideal $(t)\subset R[[t]]$. Then we have $\dot D_Y=D_Y-Y$.
\begin{lemma}\label{lm:preimage}
$j_Y^{-1}(Y)\subset Y$.
\end{lemma}
\begin{proof}
We start with
\begin{lemma}\label{lm:basechg}
Let $M$ be a finitely presented $A$-module, then the canonical morphism $M\otimes_AA[[t]]\to M[[t]]$ is an isomorphism.
\end{lemma}
\begin{proof}
Note that the functor, sending an $A$-module $M$ to $M[[t]]$ is exact. Then repeat the proof of~\cite[Prop.~10.13]{AtiyahMcdonald}.
\end{proof}

Let $U'\to U$ be finite \'etale, and surjective, set $Y':=Y\times_UU'$. Then we have a cartesian diagram (use Lemma~\ref{lm:basechg})
\[
\begin{CD}
D_{Y'} @>>> D_Y\\
@V{j_{Y'}}VV @V{j_Y}VV\\
\A^1_{U'} @>>>\A^1_U.
\end{CD}
\]
The top morphism is closed and surjective, since by Lemma~\ref{lm:basechg} it is a base change of $U'\to U$. It is also clear that the preimage of $Y$ under this morphism is $Y'$. Thus if we prove the lemma for $Y'\to U'$, we prove it for $Y\to U$. Now, applying Lemma~\ref{lm:decomp}, we can assume that $Y=\coprod_{i=1}^lY_i$, and for all $i$ the projection $Y_i$ to $U$ is an isomorphism.

It follows that for all $i$ there is $r_i\in R$ such that $Y_i$ is given by the ideal $(t-r_i)\subset R[t]$. Since for $i\ne j$ we have $Y_i\cap Y_j=\emptyset$, it follows that
\[
    (t-r_i,r_i-r_j)=(t-r_i,t-r_j)=(1),
\]
so that $r_i-r_j$ is invertible in $R$. Further, we can identify $D_Y$ with $\coprod_{i=1}^l D_U$, and one checks that the intersection of $j_Y^{-1}(Y)$ with the $i$-th component is given by the principal ideal
\[
    \left(\prod_{j=1}^l(t+r_i-r_j)\right)\subset R[[t]].
\]
However for $i\ne j$ we know that $r_i-r_j$ is invertible in $R$, so $t+r_i-r_j$ is invertible in $R[[t]]$, thus the above ideal coincides with $(t)$.
\end{proof}
The lemma implies the existence of the commutative diagram~\eqref{eq:fpqccover}.

\subsection{Gluing principal bundles} Recall that $\bG$ is an affine flat $U$-group scheme.
Let $\cE'$ be a $\bG$-bundle over $\P^1_U-Y$. Let $\hat\cE$ be a $\bG$-bundle over $D_Y$. Denote by $\cA(\cE',\hat\cE)$ the category of triples $(\cE,\tau,\sigma)$, where $\cE$ is a $\bG$-bundle over $\P^1_U$, $\tau:\cE'\to\cE|_{\P^1_U-Y}$ and $\sigma:\hat\cE\to\cE|_{D_Y}$ are isomorphisms (the restrictions are defined as the pullbacks via morphisms from diagram~\eqref{eq:fpqccover}). A morphism from $(\cE,\tau,\sigma)$ to $(\cE',\tau',\sigma')$ is an isomorphism $\cE\to\cE'$ compatible with isomorphisms $\tau$, $\tau'$, $\sigma$, and~$\sigma'$.

\begin{proposition}\label{pr:gluing}
Let $\Phi$ be the functor
\[
    \Phi:\cA(\cE',\hat\cE)\to\ISO(\hat\cE|_{\dot D_Y},\cE'|_{\dot D_Y}),
\]
sending $(\cE,\tau,\sigma)$ to $\left(\tau^{-1}|_{\dot D_Y}\right)\circ\left(\sigma|_{\dot D_Y}\right)$, where we view the right hand side as a~discrete category. Then $\Phi$ is an equivalence of categories. (Here $\ISO$ stands for the set of isomorphisms between principal bundles.)
\end{proposition}
\begin{proof}
This essentially follows from the results of~\cite{BeauvilleLaszlo} but we give some details. First of all, it is easy to see that for any two objects of $\cA$ there is at most one morphism from the first to the second. Indeed, any two such morphisms coincide on $\P_U^1-Y$ but $\P_U^1-Y$ is dense in $\P_U^1$. Thus $\Phi$ is faithful.

Next, since $\bG$-bundles and their isomorphisms can be glued in \'etale topology, the statement is \'etale local over $U$. Thus by Lemma~\ref{lm:decomp} we can assume that $Y=\coprod_{i=1}^l Y_i$ is such that for each $i$ the composition $Y_i\hookrightarrow Y\to U$ is an isomorphism. As in the proof of Lemma~\ref{lm:preimage}, we see that $Y_i$ is given by a principal ideal $(t-r_i)\subset R[t]$, $r_i\in R$.

We give a proof by induction on $l$. If $l=0$, then $Y=\emptyset$ and there is nothing to prove. Let us prove the step of induction. Let us prove that the functor $\Phi$ is essentially surjective, the fullness is proved similarly. Let $\phi\in\ISO(\hat\cE|_{\dot D_Y},\cE'|_{\dot D_Y})$.

Set $f=t-r_1$, $g=\prod_{i=2}^l(t-r_i)$. Set $Y':=\coprod_{i=2}^l Y_i$. Then
\[
    \A_U^1-Y=\spec R[t,g^{-1},f^{-1}],\qquad \A_U^1-Y'=\spec R[t,g^{-1}].
\]
Since $(t-r_1,g)$ is the unit ideal, the morphism
\[
    R[t]/t^n\to R[t,g^{-1}]/(t-r_1)^n,\qquad a(t)\mapsto a(t-r_1)
\]
is easily seen to be an isomorphism for all $n$. Passing to the limit we get an isomorphism $R[[t]]\to\widehat{R[t,g^{-1}]}$, where $\widehat{R[t,g^{-1}]}$ is the completion of $R[t,g^{-1}]$ with respect to $f$-adic topology. Now we see that the diagram
\begin{equation}\label{CD:BL}
\begin{CD}
\widehat{R[t,g^{-1}]} @>>> (\widehat{R[t,g^{-1}]})[f^{-1}]\\
@AAA @AAA \\
R[t,g^{-1}] @>>> R[t,g^{-1},f^{-1}]
\end{CD}
\end{equation}
gives rise to a diagram of morphisms of affine schemes
\begin{equation*}
\begin{CD}
D_{Y_1} @<<< \dot D_{Y_1}\\
@VVV @VVV\\
\A_U^1-Y' @<<< \A_U^1-Y
\end{CD}
\end{equation*}
and one checks that this diagram is an open subdiagram of~(\ref{eq:fpqccover}) in the obvious sense. On the other hand, diagram~\eqref{CD:BL} is an example of the diagram on page~4 of~\cite{BeauvilleLaszlo}. It follows that we can glue a flat quasi-coherent sheaf on $\A^1_U-Y$ with a flat quasi-coherent sheaf on $D_{Y_1}$, provided we are given an isomorphism of restrictions of the sheaves to $\dot D_{Y_1}$ (and we get a sheaf on $\A^1_U-Y'$). This gluing is compatible with tensor product, thus one can, in fact, glue flat affine schemes. Since the gluing of affine schemes is compatible with the product, one can glue principal bundles as well.

Let $\cE_1$ be the bundle obtained by gluing $\cE'|_{\A_U^1-Y}$ with $\hat\cE|_{D_{Y_1}}$ via $\phi|_{\dot D_{Y_1}}$. We have an isomorphism of $\cE_1$ and $\cE'$ on $\A_U^1-Y$, so we can glue them (in Zariski topology). We obtain a $\bG$-bundle $\cE'_1$ over $\P_U^1-Y'$.

It remains to apply the induction step to $\cE'_1$, $\hat\cE|_{Y'}$ and $\phi|_{Y'}$.
\end{proof}

\section{Affine Grassmannians for group schemes}\label{sect:Grassm}
This is the main section of the paper. We define affine Grassmannians for smooth affine group schemes over affine schemes. In the case, when the scheme is over a field~$k$ and the group scheme is semi-simple, we establish this affine Grassmannians as a twist of the affine Grassmannian for a $k$-group by a principal bundle, see Proposition~\ref{pr:twistbytorsor}. In this case we show that the affine Grassmannian is an ind-projective scheme and develop its Bruhat decomposition. Propositions~\ref{pr:twistedcells} and~\ref{Pr:TwistedBB} are crucial for the proof of Theorem~\ref{ThMainb}.

Affine Grassmannians were mostly studied over the field of complex numbers. Notable exceptions are~\cite{FaltingsLoopGroups} and~\cite{PappasRapoportLoopGroups}. The `family version' of affine Grassmannians was considered in~\cite[Sect.~3]{HeinlothUniformization}. However, it seems that the Bruhat decomposition was not known before in the generality needed.

Two nice reviews of affine Grassmannians are found in~\cite{SorgerLecturesBundles} (over complex numbers) and~\cite{GoertzGrassmannians}.
Below we are using the notation of Section~\ref{sect:gluing} and diagram~\eqref{eq:fpqccover}.

\subsection{Generalities on affine Grassmannians}
Let $U=\spec R$ be a connected affine scheme and $\Aff/U$ be the (big) \'etale site of affine schemes over $U$. Recall that a \emph{$U$-space} is a sheaf of sets on $\Et/U$. We can equivalently view it as a sheaf on $\Aff/U$ (see~\cite[Exp.~VII, Prop.~3.1]{SGA4-2}). Let $\bG$ be a smooth affine $U$-group scheme. Our main object of study is the \emph{affine Grassmannian\/} $Gr_\bG$. It is defined as the sheafification of the presheaf, sending an affine $U$-scheme $T$ to the set $\bG(\dot D_T)/\bG(D_T)$. (The morphism $\dot D_T\to D_T$ induces a morphism $\bG(D_T)\to\bG(\dot D_T)$. It is obvious that this morphism is injective and we identify $\bG(D_T)$ with its image.)

As in the previous section, let $Y$ be a finite and \'etale over $U$ subscheme of $\A^1_U$ (automatically closed). Assume also that $Y\ne\emptyset$, then the projection $Y\to U$ is surjective. Let $\Psi$ be the functor, sending a $U$-scheme $T$ to the set of isomorphism classes of pairs $(\cE,\tau)$, where $\cE$ is a $\bG$-bundle over $\P^1_T$, $\tau$ is its trivialization on $\P^1_T-(Y\times_UT)$.

\begin{proposition}\label{Pr:ModuliTwistedGr}
The functor $\Psi$ is canonically isomorphic to the functor sending a $U$-scheme $T$ to $Gr_\bG(Y\times_UT)$.
\end{proposition}

\begin{proof}
The proof is similar to that of~\cite[Prop.~5.3.1]{SorgerLecturesBundles}. Let $\cE$ be a $\bG$-bundle over $\P^1_T$, $\tau$ be its trivialization on $\P^1_T-(Y\times_UT)$. We need to construct a $U$-morphism from $Y\times_UT$ to $Gr_\bG$. \emph{We claim, that there is a surjective \'etale affine morphism $\phi:T'\to T$ such that the restriction of $\phi^*\cE$ to $D_{Y\times_UT'}$ is trivial.} Indeed, using Lemma~\ref{lm:decomp} and the fact that every $\bG$-bundle is \'etale locally trivial, we can choose $\phi$ such that $T'$ is affine and the restriction of $\phi^*\cE$ to $Y\times_UT'$ is trivial, that is, this restriction has a section. It remains to use
\begin{lemma}
Let $B$ be a smooth $A[[t]]$-algebra, where $A$ is a commutative ring. If there is an $A[[t]]$-homomorphism $B\to A$, then there is an $A[[t]]$-homomorphism $B\to A[[t]]$.
\end{lemma}
\begin{proof}
Denote the $A[[t]]$-homomorphism $B\to A$ by $\theta_1$. Using smoothness we can successively for all $n$ construct an $A[[t]]$-homomorphism $\theta_n:B\to A[[t]]/t^n$ such that $\theta_n\bmod t^{n-1}=\theta_{n-1}$ for all $n\ge2$. The homomorphisms $\theta_n$ give rise to an $A[[t]]$-homomorphism $\theta:B\to A[[t]]$.
\end{proof}

Let us continue with the proof of the proposition. Choose $\phi$ as above, and let~$\sigma$ be a trivialization of the restriction of $\phi^*\cE$ to $D_{Y\times_UT'}$. Then, in the notation of Proposition~\ref{pr:gluing}, $(\phi^*\cE,\phi^*\tau,\sigma)\in\cA(\cE',\hat\cE)$, where $\cE'$ and $\hat\cE$ are trivial bundles over
$\P^1_{T'}-(Y\times_UT')$ and $D_{Y\times_UT'}$ respectively. Set
\[
    \tilde\alpha:=\Phi(\phi^*\cE,\phi^*\tau,\sigma)\in\ISO(\hat\cE|_{\dot D_{Y\times_UT'}},\cE'|_{\dot D_{Y\times_UT'}})=
        \bG(\dot D_{Y\times_UT'}).
\]
Let $\alpha$ be the projection of $\tilde\alpha$ to $\bG(\dot D_{Y\times_UT'})/\bG(D_{Y\times_UT'})$, it gives rise to a $U$-morphism from $Y\times_UT'$ to $Gr_\bG$ (also denoted by $\alpha$).

\emph{We claim that $\alpha$ does not depend on a choice of $\sigma$.} Indeed, if $\sigma_1$ is a different choice of a trivialization of $\phi^*\cE$ on $D_{Y\times_UT'}$, then $\sigma_1=\sigma\circ\mu$, where $\mu\in\bG(D_{Y\times_UT'})$. But then it is easy to see that
\[
    \tilde\alpha_1:=\Phi(\phi^*\cE,\phi^*\tau,\sigma_1)=\tilde\alpha\circ\mu,
\]
so that the projection of $\tilde\alpha_1$ to $\bG(\dot D_{Y\times_UT'})/\bG(D_{Y\times_UT'})$ coincides with $\alpha$.

Now it is obvious that $\alpha$ satisfies the descent condition for the \'etale morphism $Id_Y\times_U\phi$, so it descends to a morphism from $Y\times_UT$ to $Gr_\bG$ by definition of sheafification. A standard argument shows that this morphism does not depend on the choice of $\phi$.

We have constructed a map
\[
    \delta_T:\Psi(T)\to Gr_\bG(Y\times_UT).
\]
It is straightforward to show that these maps are compatible with base changes $T'\to T$ so we get a canonical transformation $\delta$ from $\Psi$ to the functor $(T\mapsto Gr_\bG(Y\times_UT))$. It remains to show that this transformation is an isomorphism of functors, that is, for every affine $U$-scheme $T$ the map $\delta_T$ is a bijection. We will now prove the injectivity of $\delta_T$.

Assume that $\cE_1$ and $\cE_2$ are $\bG$-bundles over $\P^1_T$, $\tau_1$ and $\tau_2$ are their trivializations on $\P^1_T-(Y\times_UT)$. If $\delta_T(\cE_1,\tau_1)=\delta_T(\cE_2,\tau_2)$, then there are a surjective \'etale affine morphism $\phi:T'\to T$ and trivializations $\sigma_1$ and $\sigma_2$ of $\phi^*\cE_1$ and $\phi^*\cE_2$ on $D_{Y\times_UT'}$ such that $\Phi(\phi^*\cE_1,\phi^*\tau_1,\sigma_1)$ and $\Phi(\phi^*\cE_2,\phi^*\tau_2,\sigma_2)$ differ by an element $\mu$ of $\bG(D_{Y\times_UT'})$:
\[
    \Phi(\phi^*\cE_1,\phi^*\tau_1,\sigma_1)=\Phi(\phi^*\cE_2,\phi^*\tau_2,\sigma_2)\circ\mu.
\]
But then we have
\[
    \Phi(\phi^*\cE_1,\phi^*\tau_1,\sigma_1)=\Phi(\phi^*\cE_2,\phi^*\tau_2,\sigma_2\circ\mu).
\]
Now by Proposition~\ref{pr:gluing} there is an isomorphism $\nu:\phi^*\cE_1\to\phi^*\cE_2$ taking trivialization $\phi^*\tau_1$ to $\phi^*\tau_2$. It remains to check that $\nu$ descends to an isomorphism $\cE_1\to\cE_2$. Let $p_1$ and $p_2$ be two projections from $T'\times_TT'$ to $T'$, $q:T'\times_TT'\to T$ be the canonical morphism. We only need to check that $p_1^*\nu=p_2^*\nu$. However, both isomorphisms take $q^*\tau_1$ to $q^*\tau_2$, thus they coincide. The injectivity of $\delta_T$ is proved.

Let us prove the surjectivity of $\delta_T$. Let $\alpha\in Gr_\bG(Y\times_UT)$. By the definition of the Grassmannian, there is a surjective \'etale affine morphism $T'\to T$ such that the pullback of $\alpha$ to $Y\times_UT'$ lifts to $\tilde\alpha\in\bG(\dot D_{Y\times_UT'})$. By Proposition~\ref{pr:gluing}, there is a $\bG$-bundle $\tilde\cE$ over $\P^1_{T'}$, a trivialization $\tilde\tau$ of $\tilde\cE$ on $\P^1_{T'}-(Y\times_UT')$, and a trivialization $\tilde\sigma$ of $\tilde\cE$ on $D_{Y\times_UT'}$ such that $\Phi(\tilde\cE,\tilde\tau,\tilde\sigma)=\tilde\alpha$. Then $(\tilde\cE,\tilde\tau)\in\Psi(T')$, and, using the descent condition for $\tilde\alpha$ and Proposition~\ref{pr:gluing}, one shows that the pair $(\tilde\cE,\tilde\tau)$ descents to $(\cE,\tau)\in\Psi(T)$. It is clear that $\delta_T(\cE,\tau)=\alpha$. We see that $\delta$ is an isomorphism of functors, which completes the proof of the proposition.
\end{proof}

\begin{corollary}\label{Cor:ModuliTwistedGr}
The $U$-space $Gr_\bG$ represents the functor, sending a $U$-scheme $T$ to the set of isomorphism classes of pairs $(\cE,\tau)$, where $\cE$ is a $\bG$-bundle over $\P^1_T$, $\tau$ is a trivialization of $\cE$ on $\A^1_T$.
\end{corollary}
\begin{proof}
Take $Y=\infty\times U$ in the proposition.
\end{proof}

\subsection{Affine Grassmannians for semi-simple group schemes as twists}\label{sect:twists}
Recall that $U$ is a connected affine scheme. Now we assume that $\bG$ is a semi-simple $U$-group scheme. Let $u\in U$ be a point. Recall that the \emph{type} of $\bG$ is the isomorphism class of a root datum of $\bG_{\bar u}$, where $\bar u$ is an algebraic closure of $u$ (cf.\ Definitions~2.6.1 and~2.7 of~\cite[Exp.~XXII]{SGA3-3}). By~\cite[Exp.~XXII, Prop.~2.8]{SGA3-3} the type does not depend on the choice of $u$. Let $\bG_0$ be the split semi-simple $U$-group scheme of the same type as $\bG$ (existing by~\cite[Exp.~XXIII, Cor.~5.9]{SGA3-3}, unique up to isomorphism by~\cite[Exp.~XXIII, Cor.~5.3]{SGA3-3}). Now assume that $U$ is a $k$-scheme, where $k$ is a field. Then $\bG_0\approx G\times_kU$, where $G$ is a split $k$-group (cf.~\cite[Exp.~XXIII, Def.~5.11]{SGA3-3} and the discussion thereafter).

Let $\Aut(G):=\Aut_{k-gr}(G)$ be the scheme of $k$-group automorphisms of $G$ (cf.~\cite[Exp.~XXIV, Thm~1.3]{SGA3-3}).
Then $\bG$, as a form of $G\times_kU$, corresponds to the right principal $\Aut(G)$-bundle
\[
    \cT:=\Iso_{U-gr}(G\times_kU,\bG)
\]
over $U$; cf.~\cite[Exp.~XXIV, Cor.~1.17]{SGA3-3}. (Here $\Iso_{U-gr}(G\times_kU,\bG)$ is the scheme of isomorphisms of group schemes.)

On the other hand, $\Aut(G)$ acts on the $k$-space $Gr_G$.

\begin{proposition}\label{pr:twistbytorsor}
There is a canonical isomorphism of $U$-spaces
\[
    Gr_\bG\simeq\cT\times^{\Aut(G)}Gr_G.
\]
\end{proposition}
\begin{proof}
Let $T\to U$ be an \'etale cover trivializing the torsor $\cT$ such that $T$ is affine, choose a trivialization $\tau:G\times_kT\xrightarrow{\approx}\cT\times_UT$. Then we have
\[
    Gr_\bG\times_UT\simeq Gr_{\bG\times_UT}\simeq Gr_{G\times_kT}\simeq Gr_G\times_kT\simeq(\cT\times^{\Aut(G)}Gr_G)\times_UT.
\]
(Here the second and the last isomorphisms depend on $\tau$, while the remaining are canonical.) One checks that the resulting isomorphism does not change if we replace~$\tau$ by $\tau\sigma$, where $\sigma\in(\Aut(G))(T)$. It follows that this isomorphism descends via the cover $T\to U$.
\end{proof}

\subsection{Stratifications of ind-schemes}
The notions defined here will be used below extensively. An \emph{ind-scheme\/} over an affine scheme $U$ is a $U$-space that is a filtered direct limit of $U$-schemes in the category of $U$-spaces.
An ind-scheme $X$ over $U$ is of \emph{ind-finite type} if it can be presented as $\lim\limits_{\longrightarrow} X_i$, where $X_i$ are schemes of finite type over $U$. We say that $X$ is a \emph{strict ind-scheme\/} if it admits a presentation of the form $\lim\limits_{\longrightarrow} X_i$ such that all the corresponding morphisms $X_i\to X_j$ are closed embedding. Below we only consider strict ind-schemes of ind-finite type over $U$, and we only consider presentations $X=\lim\limits_{\longrightarrow} X_i$ with morphisms $X_i\to X_j$ being closed embeddings and with each $X_i$ of finite type over $U$. It is easy to check the following (see, e.g.,~\cite[Lemma~2.4]{GoertzGrassmannians}).
\begin{lemma}\label{lm:indschemes}
    Let $X=\lim\limits_{\longrightarrow} X_i$ be an ind-scheme over $U$ and let $Y$ be a scheme of finite type over $U$. Then every morphism from $Y$ to $X$ factors through some $X_i$.
\end{lemma}
By a \emph{subscheme\/} of $X$ we mean a subsheaf of $X$ represented by a locally closed subscheme of $X_i$ for some $i$. Using Lemma~\ref{lm:indschemes} it is easy to check that this notion does not depend on the presentation of $X$ as a limit of schemes.

\begin{definition}\label{def:stratif}
Let $X=\lim\limits_{\longrightarrow} X_i$ be a strict ind-scheme of ind-finite type over $U$. A collection $Z_\alpha$ of subschemes of $X$ is called \emph{a stratification of $X$}, if for all $i$ we have an equality of sets
\[
    X_i=\coprod_\alpha(Z_\alpha\cap X_i).
\]
\end{definition}
Again, this does not depend on a presentation of $X$ as a limit of schemes. Equivalently, the last condition means that every geometric point of $X$ factors through a unique $Z_\alpha$.

\subsection{Recollections on affine Grassmannians for split semi-simple groups}
Recall that $k$ is a field. Let $G$ be a split semi-simple $k$-group. The results in this section are well-known, when $k$ is the field of complex numbers, see~\cite[Sect.~8]{SorgerLecturesBundles}, and~\cite[Sect.~2]{BravermanFinkelberg}.
\subsubsection{$Gr_G$ is an ind-scheme}\label{sect:indschemestructure}
For positive integers $n$ and $N$, denote by $GL(n)^{(N)}$ the presheaf on $\Aff/k$ given by
\[
 GL(n)^{(N)}(\spec R)=\{(\alpha,\beta)|\alpha,\beta\in\mathrm{Mat}_{n\times n}(t^{-N}R[[t]]), \alpha\beta=1\}.
\]
(Here $\mathrm{Mat}_{n\times n}$ stands for the set of $n\times n$ matrices.) It is easy to see that $GL(n)^{(N)}$ is represented by an affine scheme, so in particular it is a sheaf. In other words, $GL(n)^{(N)}$ is the scheme of loops $\alpha$ such that both $\alpha$ and $\alpha^{-1}$ have a pole of order at most $N$. Clearly, $GL(n)^{(0)}$ is an affine $k$-group. Let $Gr^{(N)}_{GL(n)}$ be the $k$-space $GL(n)^{(N)}/GL(n)^{(0)}$. More precisely, $Gr^{(N)}_{GL(n)}$ is the sheafification of the functor
\[
    \spec R\mapsto GL(n)^{(N)}(\spec R)/GL(n)^{(0)}(\spec R).
\]
It is well known that $Gr^{(N)}_{GL(n)}$ is represented by a projective scheme (as it classifies certain lattices in $k((t))^n$, see~\cite[Prop.~8.3.2]{SorgerLecturesBundles}) and we have
\[
    Gr_{GL(n)}=\lim_{\longrightarrow} Gr^{(N)}_{GL(n)}.
\]
Next, consider a faithful $n$-dimensional representation $V$ of $G$, it gives rise to an embedding of $k$-groups $G\to GL(n)$. Then $GL(n)/G$ is an affine scheme (Indeed, $G$ is geometrically reductive by~\cite{HaboushGeomReductive}, thus the statement follows from~\cite[Sect.~8]{NagataInvariants}, see also~\cite[Corollary]{NisnevichAffineHomogeneous}). Therefore we get a closed embedding $Gr_G\to Gr_{GL(n)}$, see the lemma in the proof of~\cite[Thm.~4.5.1]{BeilinsonDrinfeldHitchin}.
Let $Gr_G^{(N)}$ be the preimage of $Gr^{(N)}_{GL(n)}$ under this closed embedding; it is a closed subscheme in $Gr^{(N)}_{GL(n)}$. We have
\[
    Gr_G=\lim_{\longrightarrow} Gr_G^{(N)}.
\]
In particular, we see that $Gr_G$ is an ind-scheme.

\subsubsection{$L^+G$-orbits}
Let $L^+G$ be the jet scheme of $G$; it represents the functor $\spec R\mapsto G(R[[t]])$. Thus $L^+G$ is an affine $k$-group scheme of infinite type. Note that $L^+GL(n)=GL(n)^{(0)}$. Clearly, $L^+G$ acts on $Gr_G$; we want to describe the orbits.

Note that the action of $L^+G$ preserves the schemes $Gr_G^{(N)}$. Moreover, let $L^+_{(M)}G$ denote the group scheme representing the functor $\spec R\mapsto G(R[[t]]/t^M)$. Then we have a homomorphism $L^+G\to L^+_{(M)}G$ and it is easy to see that the action of $L^+G$ on $Gr_G^{(N)}$ factors through $L^+_{(M)}G$ for $M\gg N$. Since $L^+_{(M)}G$ is a group scheme of finite type, the orbits of $L^+G$ on $Gr_G^{(N)}$ are subschemes of $Gr_G$ (locally closed).

Let $X_*=X_*(G)$ be the lattice of co-characters of $G$. For any choice of a Borel subgroup $B\subset G$ and a split maximal torus $T\subset B$ we get an identification
\begin{equation*}
    X_*=\Hom(\Gm,T)\subset T\bigl(k((t))\bigr).
\end{equation*}
For $\lambda\in X_*$ denote by $t^\lambda$ the corresponding element of $T\bigl(k((t))\bigr)$. Abusing notation, we also denote by $t^\lambda$ the projection to $Gr_G(k)$ of
\[
    t^\lambda\in T\bigl(k((t))\bigr)\subset G\bigl(k((t))\bigr).
\]
Denote by $Gr_G^\lambda$ the $L^+G$-orbit of $t^\lambda$.

Since any two pairs $(T,B)$ are conjugate, $Gr_G^\lambda$ does not depend on the choices of~$T$ and~$B$. The following proposition is well-known, when $k$ is the field of  complex numbers. In general, it is easily reduced to the case, when $k$ is algebraically closed, where it follows from~\cite[Prop.~8]{HainesRapoportParahoric} (see also~\cite{BruhatTits}).

\begin{proposition}\label{pr:cells}
We have $Gr_G^\lambda=Gr_G^\mu$ if and only if $\lambda$ and $\mu$ are $W$-conjugate, where $W$ is the Weyl group of $G$. Further, we have a stratification (see Definition~\ref{def:stratif})
\begin{equation*}
    Gr_G=\coprod_{\lambda\in X_*/W}Gr_G^\lambda.
\end{equation*}
\end{proposition}
(Slightly abusing notation, we denote by $Gr_G^\lambda$ the orbit $Gr_G^{\tilde\lambda}$, where $\tilde\lambda$ is any lift of $\lambda$ to $X_*$.)
\subsubsection{The structure of $Gr_G^\lambda$}
Note that $G$ is a subgroup of $L^+G$. Denote the $G$-orbit of $t^\lambda$ by $F_G^\lambda$. We have an evaluation morphism $L^+G\to G$ left inverse to the embedding $G\hookrightarrow L^+G$. One checks that the evaluation morphism takes the stabilizer of $t^\lambda$ in $L^+G$ to the stabilizer of $t^\lambda$ in $G$. This is used to construct a morphism
\begin{equation}\label{projection}
    Gr_G^\lambda\to F_G^\lambda.
\end{equation}
\begin{remarks}\label{rm:smoothmap}
\stepzero\noindstep
    While this morphism is easy to construct, when $k$ is the field of characteristic zero, some care should be taken in finite characteristic; let us indicate the main steps. First of all, one can replace $L^+G$ by $L^+_{(M)}G$, where $M$ is large enough. Next, one needs to show that the stabilizer of $t^\lambda$ in $L^+_{(M)}G$ is smooth (so that $Gr_G^\lambda$ is a categorical quotient of $L^+_{(M)}G$). To this end, one first calculates the Lie algebra $\fk$ of this stabilizer. Then one constructs a subgroup $K\subset L^+_{(M)}G$ with the Lie algebra $\fk$. It remains to show that $K$ is contained in the stabilizer of $t^\lambda$, which is done by reducing the statement to the case $G=GL(n)$ via a faithful representation of $G$.

\noindstep\label{rm:smoothitem}
    On can show that this morphism is smooth and the fibers are isomorphic to vector spaces.
\end{remarks}

It is easy to see that the $G$-stabilizer $P_s$ of any $s\in F_G^\lambda$ is a parabolic subgroup in $G_{k(s)}=G\times_k\spec k(s)$, where, as usual, $k(s)$ is the residue field of $s$. This is a parabolic subgroup of type $\lambda$, that is, the Weyl group of a Levi factor of $P_s$ is the stabilizer of $\lambda$ in $W$.
\begin{lemma}
The map $s\mapsto P_s$ is an isomorphism from $F_G^\lambda$ to the flag variety of parabolic subgroups of type $\lambda$ in $G$.
\end{lemma}
We emphasize that this isomorphism does not depend on a choice of a torus $T$ and a Borel subgroup $B$.
\begin{proof}
After a choice of a Borel subgroup and a split maximal torus in $G$, both $F_G^\lambda$ and the flag variety get identified with $G/P_\lambda$, where $P_\lambda$ is the stabilizer of $t^\lambda$ in $G$.
\end{proof}
Thus we can (and always will) identify $F_G^\lambda$ with the flag variety of parabolic subgroups of type $\lambda$ in $G$, so that morphism~\eqref{projection} becomes a morphism from $Gr_G^\lambda$ to a partial flag variety of $G$.

\subsubsection{The action of automorphisms}
Recall from~\cite[Exp.~XXIV]{SGA3-3} (see especially Th.~1.3) that we have an exact sequence
\begin{equation}\label{eq:exseq}
    1\to\Inn(G)\to\Aut(G)\to\Out(G)\to1.
\end{equation}
Here $\Aut(G):=\Aut_{k-gr}(G)$; $\Inn(G)$ is the group of inner automorphisms of $G$, it is isomorphic to the quotient of $G$ by its center; $\Out(G)$ is the group of outer automorphisms, it is a finite $k$-group. Note that $\Aut(G)$ acts on $Gr_G$ but the strata of Proposition~\ref{pr:cells} are only $\Inn(G)$ invariant, $\Out(G)$ permutes them. We want to develop a coarser stratification with $\Aut(G)$-invariant strata.

First of all, $\Out(G)$ acts on a root datum of $G$ and, thus, on the Weyl group $W=W(G)$ and on the co-character lattice $X_*=X_*(G)$. Thus we get a semidirect product $W\leftthreetimes\Out(G)$ and its action on $X_*$. For $\hat\lambda\in X_*/(W\leftthreetimes\Out(G))$, we denote by $\Orb(\hat\lambda)\subset X_*/W$ the $\Out(G)$-orbit of any lift of $\hat\lambda$ to $X_*/W$. Note that $\Inn(G)$ preserves the orbits $Gr_G^\lambda$, while $\Aut(G)$ permutes them according to the action of $\Out(G)$ on $X_*/W$. For $\hat\lambda\in X_*/(W\leftthreetimes\Out(G))$ set
\[
    Gr^{\hat\lambda}_G:=\bigcup_{\lambda\in\Orb(\hat\lambda)}Gr^\lambda_G\subset Gr_G.
\]
Clearly, this is $\Aut(G)$-invariant locally closed subscheme of $Gr_G$. Moreover, if $\lambda_1,\lambda_2\in\Orb(\hat\lambda)$, then $Gr^{\lambda_1}_G\approx Gr^{\lambda_2}_G$, so these orbits have the same dimension. It follows that $Gr^{\lambda_1}_G$ cannot lie in the closure of $Gr^{\lambda_2}_G$, and we have (scheme theoretically)
\[
    Gr^{\hat\lambda}_G=\coprod_{\lambda\in\Orb(\hat\lambda)}Gr^\lambda_G.
\]
By Proposition~\ref{pr:cells} we have
\begin{equation}\label{eq:cells}
    Gr_G=\coprod_{\hat\lambda\in X_*/(W\leftthreetimes\Out(G))}Gr^{\hat\lambda}_G.
\end{equation}
Next, set
\[
    F_G^{\hat\lambda}:=\bigcup_{\lambda\in\Orb(\hat\lambda)}F_G^\lambda=
    \coprod_{\lambda\in\Orb(\hat\lambda)}F_G^\lambda.
\]
The next lemma follows from definitions
\begin{lemma}\label{L:Autequiv}
For each $\hat\lambda\in X_*/(W\leftthreetimes\Out(G))$ the morphisms~\eqref{projection} give rise to an $\Aut(G)$-equivariant morphism
\begin{equation}\label{eq:mortopar}
    Gr_G^{\hat\lambda}\to F_G^{\hat\lambda}.
\end{equation}
\end{lemma}

By a slight abuse of notation we denote by $1\in Gr_G(k)$ the image of the unity of $G\bigl(k((t))\bigr)$. Clearly $Gr_G^0=F_G^0=\{1\}$. For all $\lambda\ne0$, $F_G^\lambda$ has positive dimension, that is, $F_G^\lambda$ classifies some proper parabolic subgroups.

\subsection{Group scheme affine Grassmannians are ind-schemes}
Assume that we are in the situation of Proposition~\ref{pr:twistbytorsor}. Recall that in Section~\ref{sect:indschemestructure} we defined the closed subschemes $Gr_G^{(N)}\subset Gr_G$ depending on a faithful $n$-dimensional representation $V$ of $G$. Note that \emph{we may assume that $\Out_k(G)$ acts on $GL(n)$ by automorphisms and the embedding is equivariant}. Indeed, it is enough to replace
$V$ by $\oplus_{\sigma\in\Out(G)}V^\sigma$, where $V^\sigma$ is the twist of the representation $V$ by $\sigma$ (we use a splitting of~\eqref{eq:exseq}).

Then it is easy to check that $Gr_G^{(N)}$ is $\Aut(G)$-invariant, so, by Proposition~\ref{pr:scheme}, we can form an associated $U$-scheme
\[
    Gr_\bG^{(N)}:=\cT\times^{\Aut(G)}Gr_G^{(N)}
\]
(indeed, $\Aut(G)$ is affine of finite type by~\cite[Exp.~XXIV, Cor.~1.8]{SGA3-3}). The following proposition shows that $Gr_\bG$ is an ind-scheme.
\begin{proposition}\label{pr:indscheme}
\[
    Gr_\bG=\lim\limits_{\longrightarrow} Gr_\bG^{(N)}.
\]
\end{proposition}
\begin{proof}
It is easy to construct a morphism of $U$-spaces
\[
    \lim_{\longrightarrow} Gr_\bG^{(N)}\to\cT\times^{\Aut(G)}\left(\lim_{\longrightarrow} Gr_G^{(N)}\right)=Gr_\bG.
\]
It is enough to check that it is an isomorphism after a surjective \'etale base change. But we can choose this base change $T\to U$ so that both spaces become ${Gr_G\times_kT}$.
\end{proof}

\subsection{Decompositions of group scheme affine Grassmannians}
Assume again that we are in the situation of Proposition~\ref{pr:twistbytorsor}. The stratification of Proposition~\ref{pr:cells} does not always give rise to a stratification of $Gr_\bG$ but the coarser stratification~\eqref{eq:cells} does as we presently explain. Recall that $Gr_G^{\hat\lambda}\subset Gr_G$ is $\Aut(G)$-invariant, so we can set
\[
    Gr^{\hat\lambda}_\bG:=\cT\times^{\Aut(G)}Gr^{\hat\lambda}_G.
\]
By Proposition~\ref{pr:twistbytorsor} $Gr^{\hat\lambda}_\bG$ is a locally closed subscheme of $Gr_\bG$.
\begin{proposition}\label{pr:twistedcells}
We have a stratification (in the sense of Definition~\ref{def:stratif})
\begin{equation*}
    Gr_\bG=\coprod_{\hat\lambda\in X_*/(W\leftthreetimes\Out(G))}Gr^{\hat\lambda}_\bG.
\end{equation*}
\end{proposition}
\begin{proof}
This follows from~\eqref{eq:cells}.
\end{proof}

Set also
\[
    F^{\hat\lambda}_\bG:=\cT\times^{\Aut(G)}F_G^{\hat\lambda}.
\]
Note that $F^{\hat\lambda}_\bG$ is a union of connected components of a scheme classifying parabolic subgroup schemes of $\bG$, see~~\cite[Exp.~XXVI, Sect.~3]{SGA3-3}.

\begin{proposition}\label{Pr:TwistedBB}
For any $\hat\lambda\in X_*/(W\leftthreetimes\Out(G))$ there is a $U$-morphism
\begin{equation}\label{eq:bblimit}
    Gr^{\hat\lambda}_\bG\to F^{\hat\lambda}_\bG.
\end{equation}
\end{proposition}
\begin{proof}
Twist~\eqref{eq:mortopar} by $\cT$.
\end{proof}
\begin{remark}
One can show that the above morphism is smooth with fibers isomorphic to a $k$-vector space, cf.~Remark~\ref{rm:smoothmap}\eqref{rm:smoothitem}.
\end{remark}

\section{A proof of Theorem~\ref{ThMainb}}\label{sect:exoticism}
We will use notation from the statement of the theorem. Let
\[
    \phi:\bG\times_U(\P^1_U-Y)\to\cE|_{\P^1_U-Y}
\]
be a trivialization. Then by Proposition~\ref{Pr:ModuliTwistedGr} $(\cE,\phi)$ gives rise to a $U$-morphism $s:Y\to Gr_\bG$. Let $\omega$ be a generic point of $Y$; let $\hat\lambda\in X_*/(W\leftthreetimes\Out(G))$ be such that $s(\omega)\in Gr^{\hat\lambda}_\bG$, see Proposition~\ref{pr:twistedcells}. Composing $s$ with~\eqref{eq:bblimit}, we get a $U$-morphism $\omega\to F^{\hat\lambda}_\bG$, that is, a choice of a parabolic subgroup in $\bG\times_U\omega$. However, we assumed that $\bG$ is anisotropic at $\omega$, so this parabolic subgroup cannot be proper, which means that $\hat\lambda=0$ (see the paragraph after Lemma~\ref{L:Autequiv}). But then, since $Gr^0_\bG$ is a closed subscheme of $Gr_\bG$, we see that $s(\bar\omega)\subset Gr^0_\bG$, where $\bar\omega$ is the Zariski closure of $\omega$. Since this is true for all generic points $\omega$ of $Y$, we see that $s(Y)\subset Gr^0_\bG$, that is, $s$ is the trivial section of $Gr_\bG$. It follows that $\cE$ is trivial. Theorem~\ref{ThMainb} is proved. \hfill\qed

\section{Extending sections of affine Grassmannians}\label{sect:constructing}
\subsection{Proof of Theorem~\ref{ThMaina}}\label{ProofThMaina} We use the notation from the theorem statement. We view $Z_u$ as a $U$-scheme via the composition morphism $Z_u\to u\to U$. The proof is based on the following extension property.
\begin{proposition}\label{pr:GrSurj}
The restriction morphism $Gr_\bG(Z)\to Gr_\bG(Z_u)$ is surjective.
\end{proposition}
\begin{proof}[Derivation of Theorem~\ref{ThMaina} from the proposition]
Note first that the exact sequence from~\cite[Cor.~3.10(a)]{GilleTorseurs} shows that $E$ is locally trivial in Zariski topology on $\P^1_u$. Let $v$ be a $k(u)$-rational point of $Z_u$, then $E$ is trivial on $\P^1_u-v\approx\A^1_u$ by the second part of~\cite[Cor.~3.10(a)]{GilleTorseurs}. Thus $E$ is also trivial on $\P^1_u-Z_u$, let $\tau$ be a trivialization. The pair $(E,\tau)$ can be viewed as a $Z_u$-point of $Gr_\bG$ due to Proposition~\ref{Pr:ModuliTwistedGr}. By the proposition, this $Z_u$-point can be extended to a $Z$-point of $Gr_\bG$, that is, to a pair $(\cE,\tilde\tau)$, where $\cE$ is a $\bG$-bundle over $\P^1_U$, $\tilde\tau$ is its trivialization away from $Z$. Clearly, $\cE$ satisfies the requirements of Theorem~\ref{ThMaina}.
\end{proof}

The following argument is similar to~\cite[Sect.~5.7]{FedorovPanin}. By our assumption on the group scheme $\bG_Z=\bG\times_UZ$ we can and will choose a parabolic subgroup scheme $\bP^+\subset\bG_Z$ such that the restriction of $\bP^+$ to each connected component of $Z$ is a proper parabolic subgroup scheme in the restriction of $\bG_Z$ to this component of $Z$.

Since $Z$ is an affine scheme, by~\cite[Exp.~XXVI, Cor.~2.3, Thm.~4.3.2(a)]{SGA3-3} there is an opposite to $\bP^+$ parabolic subgroup scheme $\bP^-$ in $\bG_Z$. Let $\bU^+$ be the unipotent radical of $\bP^+$, and let $\bU^-$ be the unipotent radical of $\bP^-$.

\begin{definition}\label{EYi}
We will write $\bE$ for the functor, sending a $Z$-scheme $T$ to the subgroup $\bE(T)$ of the group $\bG_Z(T)=\bG(T)$ generated by the subgroups $\bU^+(T)$ and $\bU^-(T)$ of the group $\bG_Z(T)=\bG(T)$.
\end{definition}

\begin{lemma}\label{lm:surjectivity}
The functor $\bE$ has the property that for every closed subscheme $S$ in an affine $Z$-scheme $T$ the induced map $\bE(T)\to\bE(S)$ is surjective.
\end{lemma}
\begin{proof}
The restriction maps $\bU^\pm(T)\to\bU^\pm(S)$ are surjective, since $\bU^\pm$ are isomorphic to vector bundles as $Z$-schemes
(see~\cite[Exp.~XXVI, Cor.~2.5]{SGA3-3}).
\end{proof}

Recall that $D_Z$ is the formal disc over $Z$, $\dot D_Z$ is the punctured formal disc over~$Z$. We view $\dot D_Z$ as a $Z$-scheme via the projection. Thus its closed subscheme $\dot D_{Z_u}$ is also a $Z$-scheme. Hence $\bE(\dot D_Z)$ and $\bE(\dot D_{Z_u})$ make sense.

\begin{proof}[Proof of Proposition~\ref{pr:GrSurj}]
Consider the diagram
\[
\begin{CD}
\bE(\dot D_Z) @>>> \bE(\dot D_{Z_u})\\
@VVV @VVV\\
Gr_\bG(Z) @>>> Gr_\bG(Z_u).
\end{CD}
\]
The left vertical arrow is the composition
\[
\bE(\dot D_Z)\hookrightarrow\bG(\dot D_Z)
\to\bG(\dot D_Z)/\bG(D_Z)\to Gr_\bG(Z).
\]
The right vertical arrow is defined similarly. The top horizontal map is surjective by Lemma~\ref{lm:surjectivity}. Thus it is enough to show that the right vertical arrow is surjective. Let $(F,\tau)\in Gr_\bG(Z_u)$, that is, $F$ is a $\bG_u$-bundle over $\P^1_u$, $\tau$ is its trivialization on $\P^1_u-Z_u$. Recall that we have a canonical morphism $D_{Z_u}\to\P^1_u$ (cf.~diagram~\eqref{eq:fpqccover}).
The exact sequence from~\cite[Cor.~3.10(a)]{GilleTorseurs} shows that $F$ is locally trivial in Zariski topology on $\P^1_u$. Therefore it is trivial on $D_{Z_u}$; let $\sigma$ be a trivialization. Thus, in the notation of Proposition~\ref{pr:gluing}, $(F,\tau,\sigma)\in\cA(F',\hat F)$, where $F'$ and $\hat F$ are trivial $\bG_u$-bundles over $\P^1_u-Z_u$ and $D_{Z_u}$ respectively. Then
\[
    \Phi(F,\tau,\sigma)\in\bG(\dot D_{Z_u})=\prod_{v\in Z_u}\bG(\dot D_v)
\]
(note that $Z_u$ is a finite scheme). We need the following lemma.
\begin{lemma}
Let $k$ be an infinite field, $H$ be a simple simply-connected $k$-group, $P^\pm\subset H$ be opposite parabolic subgroups of $H$ (defined over $k$). Let $U^\pm$ be unipotent radicals of $P^\pm$. For a $k$-scheme $T$ denote by $E(T)$ the subgroup of $H(T)$ generated by $U^\pm(T)$. We have
\[
    H\bigl(k((t))\bigr)=E\bigl(k((t))\bigr)H\bigl(k[[t]]\bigr).
\]
\end{lemma}
\begin{proof}
Let $\alpha\in H\bigl(k((t))\bigr)$. By~\cite[Thm.~3.4]{GilleTorseurs} applied to $\alpha^{-1}$ we can write $\alpha=\beta\gamma$, where $\beta\in H(k(t))$, $\gamma\in H\bigl(k[[t]]\bigr)$. By Propositions~8.4 and~8.5 of~\cite{PaninStavrovaVavilov} we can write  $\beta$ as $\beta'\beta''$, where $\beta'\in E(k(t))\subset E\bigl(k((t))\bigr)$, $\beta''\in H\subset H\bigl(k[[t]]\bigr)$.
\end{proof}
Applying this lemma to each point of $Z_u$, we can write $\Phi(F,\tau,\sigma)$ as $\beta\gamma$, where $\beta\in\bE(\dot D_{Z_u})$, $\gamma\in\bG_u(D_{Z_u})$. It follows from the proof of Proposition~\ref{Pr:ModuliTwistedGr} that under the isomorphism of this proposition the projection of $\Phi(F,\tau,\sigma)=\beta\gamma$ to $Gr_{\bG_u}$ corresponds to $(F,\tau)$. Then the image of $\beta$ under the right vertical map in the above diagram is also $(F,\tau)$.
\end{proof}

\subsection{Proof of Theorem~\ref{th:exot}}\label{ProofThExot}
Let $\omega$ be the generic point of $U$, let $\cF$ be a principal $G$-bundle over $U$ such that $\cF$ cannot be reduced to a proper parabolic subgroup of $G$ at $\omega$. By Proposition~\ref{pr:parred} the group scheme $\Aut(\cF)$ is anisotropic at the generic point. Since we are working over the algebraically closed field $k$, $\Aut(\cF)$ is isotropic at the closed point of $U$. Thus by Corollary~\ref{cor:example} there exists a non-trivial $\Aut(\cF)$-bundle $\cE$ over $\A^1_U$ such that $\cE$ is trivial on $\A^1_U-Z$ for a certain $Z$ finite and \'etale over $U$. This proves part~\eqref{th:exotB}.

To prove part~\eqref{th:exotA} consider the associated scheme (cf.~\eqref{eq:notation})
\[
    \cE':=\cF\times^{\Aut(\cF)}\cE=p_U^*\cF\times^{p_U^*\Aut(\cF)}\cE.
\]
This is the required $G$-bundle. Indeed, if it was isomorphic to $p_U^*\cF_0$, then, as it is easy to check
\[
    \cE\approx\Iso_G(p_U^*\cF,p_U^*\cF_0)\approx p_U^*(\Iso_G(\cF,\cF_0)),
\]
where $\Iso_G$ is the scheme of isomorphisms of $G$-bundles. But, as it was explained in the last paragraph of Section~\ref{sect:PrincipalOverLines}, $\cE$ is not isomorphic to the pullback of an $\Aut(\cF)$-bundle over $U$. Theorem~\ref{th:exot} is proved.

\subsection{Modifications and a proof of Theorem~\ref{Th:B}\eqref{thprB:LineB}}\label{sect:modifications}
In the situation of Theorem~\ref{Th:B} we call \emph{a modification of $\cE$ at $Y$} a $\bG$-bundle $\cF$ over $\P_U^1$ together with an isomorphism
\[
    \cF|_{\P^1_U-Y}\xrightarrow{\cong}\cE|_{\P^1_U-Y}.
\]
We can assume that $Y\cap Z=\emptyset$ (cf. Remark~2 after~\cite[Thm.~3]{FedorovPanin}). Let us choose a trivialization of $\cE$ on $\P^1_U-Z$.
\begin{proposition}
The functor, sending a $U$-scheme $T$ to the set of isomorphism classes of modifications of $\phi^*\cE$ at $Y\times_UT$, is isomorphic to the functor, sending $T$ to $Gr_\bG(Y\times_UT)$.
\end{proposition}
The proof is a slight generalization of the proof of Proposition~\ref{Pr:ModuliTwistedGr} (note that $\cE$ is trivialized in a Zariski neighborhood of $Y$).

\begin{proof}[Sketch of proof of Theorem~\ref{Th:B}\eqref{thprB:LineB}] Similarly to the proof of Theorem~\ref{ThMaina} we can find a modification $(F,\tau)$ of $\cE_u$ at $Y_u$ such that $F$ is a trivial bundle over $\P_u^1$. Using the above proposition and Proposition~\ref{pr:GrSurj} we can find a modification $(\cF,\tilde\tau)$ of $\cE$  at $Y$ such that $\cF_u\approx F$. Applying Theorem~\ref{Th:B}\eqref{thprB:LineA} with $Z\cup Y$ instead of $Z$, we see that $\cF$ is trivial. Thus
\[
\cE|_{\P_U^1-Y}\approx\cF|_{\P_U^1-Y}
\]
is trivial.
\end{proof}

\bibliographystyle{alphanum}
\bibliography{exotic}

\begin{thebibliography}{Gro3}

\bibitem[AK]{AltmanKleiman}
Allen Altman and Steven Kleiman.
\newblock {\em Introduction to {G}rothendieck duality theory}.
\newblock Lecture Notes in Mathematics, Vol. 146. Springer-Verlag, Berlin,
  1970.

\bibitem[AM]{AtiyahMcdonald}
Michael~F. Atiyah and Ian~G. Macdonald.
\newblock {\em Introduction to commutative algebra}.
\newblock Addison-Wesley Publishing Co., Reading, Mass.-London-Don Mills, Ont.,
  1969.

\bibitem[Bas]{BassKTheory}
Hyman Bass.
\newblock Some problems in ``classical'' algebraic {$K$}-theory.
\newblock In {\em Algebraic {$K$}-theory, {II}: ``{C}lassical'' algebraic
  {$K$}-theory and connections with arithmetic ({P}roc. {C}onf., {B}attelle
  {M}emorial {I}nst., {S}eattle, {W}ash., 1972)}, pages 3--73. Lecture Notes in
  Math., Vol. 342. Springer, Berlin, 1973.

\bibitem[BD]{BeilinsonDrinfeldHitchin}
Alexander Beilinson and Vladimir Drinfeld.
\newblock Quantization of {H}itchin's integrable system and {H}ecke
  eigensheaves.
\newblock
  http://www.math.uchicago.edu/$\sim$mitya/langlands/hitchin/BD-hitchin.pdf,
  1991.

\bibitem[BF]{BravermanFinkelberg}
Alexander Braverman and Michael Finkelberg.
\newblock Pursuing the double affine {G}rassmannian. {I}. {T}ransversal slices
  via instantons on {$A\sb k$}-singularities.
\newblock {\em Duke Math. J.}, 152(2):175--206, 2010.

\bibitem[BL]{BeauvilleLaszlo}
Arnaud Beauville and Yves Laszlo.
\newblock Un lemme de descente.
\newblock {\em C. R. Acad. Sci. Paris S\'er. I Math.}, 320(3):335--340, 1995.

\bibitem[Bor]{BorelLinAlgGrps}
Armand Borel.
\newblock {\em Linear algebraic groups}, volume 126 of {\em Graduate Texts in
  Mathematics}.
\newblock Springer-Verlag, New York, second edition, 1991.

\bibitem[BT1]{Borel-Tits2}
Armand Borel and Jacques Tits.
\newblock Compl\'ements \`a l'article: ``{G}roupes r\'eductifs''.
\newblock {\em Inst. Hautes \'Etudes Sci. Publ. Math.}, (41):253--276, 1972.

\bibitem[BT2]{BruhatTits}
Fran\c{c}ois Bruhat and Jacques Tits.
\newblock Groupes r\'eductifs sur un corps local.
\newblock {\em Inst. Hautes \'Etudes Sci. Publ. Math.}, (41):5--251, 1972.

\bibitem[DG]{SGA3-3}
Michel Demazure and Alexander Grothendieck.
\newblock {\em Sch\'emas en groupes. {III}: {S}tructure des sch\'emas en
  groupes r\'eductifs}.
\newblock S\'eminaire de G\'eom\'etrie Alg\'ebrique du Bois Marie 1962/64 (SGA
  3). Dirig\'e par M. Demazure et A. Grothendieck. Lecture Notes in
  Mathematics, Vol. 153. Springer-Verlag, Berlin, 1970.

\bibitem[Fal]{FaltingsLoopGroups}
Gerd Faltings.
\newblock Algebraic loop groups and moduli spaces of bundles.
\newblock {\em J. Eur. Math. Soc. (JEMS)}, 5(1):41--68, 2003.

\bibitem[FP]{FedorovPanin}
Roman Fedorov and Ivan Panin.
\newblock A proof of the {G}rothendieck--{S}erre conjecture on principal
  bundles over regular local rings containing infinite fields.
\newblock {\em Publications math\'ematiques de l'IH\'ES}, 122(1):169--193,
  2015.

\bibitem[FR]{FerrandRaynaud}
Daniel Ferrand and Michel Raynaud.
\newblock Fibres formelles d'un anneau local noeth\'erien.
\newblock {\em Ann. Sci. \'Ecole Norm. Sup. (4)}, 3:295--311, 1970.

\bibitem[Gil1]{GilleTorseurs}
Philippe Gille.
\newblock Torseurs sur la droite affine.
\newblock {\em Transform. Groups}, 7(3):231--245, 2002.

\bibitem[Gil2]{GilleErratum}
Philippe Gille.
\newblock Errata: ``{T}orsors on the affine line'' ({F}rench) [{T}ransform.
  {G}roups {\bf 7} (2002), 231--245].
\newblock {\em Transform. Groups}, 10(2):267--269, 2005.

\bibitem[G{\"o}r]{GoertzGrassmannians}
Ulrich G{\"o}rtz.
\newblock Affine {S}pringer fibers and affine {D}eligne-{L}usztig varieties.
\newblock In {\em Affine flag manifolds and principal bundles}, Trends Math.,
  pages 1--50. Birkh\"auser/Springer Basel AG, Basel, 2010.

\bibitem[Gro1]{GrothendieckTorsion}
Alexander Grothendieck.
\newblock Torsion homologique et sections rationnelles.
\newblock In {\em Anneaux de Chow et applications, S\'eminaire Claude
  Chevalley}, number~3. Paris, 1958.

\bibitem[Gro2]{GrothendieckBrauer2}
Alexander Grothendieck.
\newblock Le groupe de {B}rauer. {II}. {T}h\'eorie cohomologique.
\newblock In {\em Dix {E}xpos\'es sur la {C}ohomologie des {S}ch\'emas}, pages
  67--87. North-Holland, Amsterdam, 1968.

\bibitem[Gro3]{FGA1}
Alexander Grothendieck.
\newblock Technique de descente et th\'eor\`emes d'existence en g\'eometrie
  alg\'ebrique. {I}. {G}\'en\'eralit\'es. {D}escente par morphismes
  fid\`element plats.
\newblock In {\em S\'eminaire {B}ourbaki, {V}ol.\ 5, {E}xp.\ {N}o.\ 190.},
  pages 299--327. Soc. Math. France, Paris, 1995.

\bibitem[Hab]{HaboushGeomReductive}
William~J. Haboush.
\newblock Reductive groups are geometrically reductive.
\newblock {\em Ann. of Math. (2)}, 102(1):67--83, 1975.

\bibitem[Hei]{HeinlothUniformization}
Jochen Heinloth.
\newblock Uniformization of {$\mathcal G$}-bundles.
\newblock {\em Math. Ann.}, 347(3):499--528, 2010.

\bibitem[HR]{HainesRapoportParahoric}
Thomas Haines and Michael Rapoport.
\newblock Appendix: On parahoric subgroups.
\newblock {\em Advances in Mathematics}, 219(1):188--198, 2008.

\bibitem[Lin]{LindelOnBassQuillen}
Hartmut Lindel.
\newblock On the {B}ass-{Q}uillen conjecture concerning projective modules over
  polynomial rings.
\newblock {\em Invent. Math.}, 65(2):319--323, 1981/82.

\bibitem[Nag]{NagataInvariants}
Masayoshi Nagata.
\newblock Invariants of a group in an affine ring.
\newblock {\em J. Math. Kyoto Univ.}, 3:369--377, 1963/1964.

\bibitem[Nis]{NisnevichAffineHomogeneous}
Yevsey Nisnevich.
\newblock Affine homogeneous spaces and finite subgroups of arithmetic groups
  over function fields.
\newblock {\em Functional Analysis and Its Applications}, 11(1):64--66, 1977.

\bibitem[OP]{OjangurenPanin2}
Manuel Ojanguren and Ivan Panin.
\newblock Rationally trivial {H}ermitian spaces are locally trivial.
\newblock {\em Math. Z.}, 237(1):181--198, 2001.

\bibitem[Pan]{PaninPurity}
Ivan Panin.
\newblock {On Grothendieck--Serre's conjecture concerning principal G-bundles
  over reductive group schemes:II}.
\newblock {\em ArXiv e-prints, 0905.1423}, April 2013.

\bibitem[Par]{ParimalaFailure}
S.~Parimala.
\newblock Failure of a quadratic analogue of {S}erre's conjecture.
\newblock {\em Amer. J. Math.}, 100(5):913--924, 1978.

\bibitem[PP]{PaninPetrovPurity}
Ivan Panin and Victor Petrov.
\newblock {Rationally isotropic exceptional projective homogeneous varieties
  are locally isotropic}.
\newblock {\em ArXiv e-prints, 1210.6695}, October 2012.

\bibitem[PR]{PappasRapoportLoopGroups}
Georgios Pappas and Michael Rapoport.
\newblock Twisted loop groups and their affine flag varieties.
\newblock {\em Adv. Math.}, 219(1):118--198, 2008.
\newblock With an appendix by T. Haines and Rapoport.

\bibitem[PSV]{PaninStavrovaVavilov}
I.~Panin, A.~Stavrova, and N.~Vavilov.
\newblock On {G}rothendieck-{S}erre's conjecture concerning principal
  {$G$}-bundles over reductive group schemes: {I}.
\newblock {\em Compos. Math.}, 151(3):535--567, 2015.

\bibitem[Qui]{QuillenOnSerre}
Daniel Quillen.
\newblock Projective modules over polynomial rings.
\newblock {\em Invent. Math.}, 36:167--171, 1976.

\bibitem[Rag]{RaghunathanOnAffineSpace}
Madabusi~S. Raghunathan.
\newblock Principal bundles on affine space and bundles on the projective line.
\newblock {\em Math. Ann.}, 285(2):309--332, 1989.

\bibitem[RR]{RagunathanRamanthan}
Madabusi~S. Raghunathan and Annamalai Ramanathan.
\newblock Principal bundles on the affine line.
\newblock {\em Proc. Indian Acad. Sci. Math. Sci.}, 93(2-3):137--145, 1984.

\bibitem[Ser]{SerreFibres}
Jean-Pierre Serre.
\newblock Espaces fibr\'es alg\'ebrique.
\newblock In {\em Anneaux de Chow et applications, S\'eminaire Claude
  Chevalley}, number~3. Paris, 1958.

\bibitem[SGA]{SGA4-2}
{\em Th\'eorie des topos et cohomologie \'etale des sch\'emas. {T}ome 2}.
\newblock Lecture Notes in Mathematics, Vol. 270. Springer-Verlag, Berlin-New
  York, 1972.
\newblock S{\'e}minaire de G{\'e}om{\'e}trie Alg{\'e}brique du Bois-Marie
  1963--1964 (SGA 4), Dirig{\'e} par M. Artin, A. Grothendieck et J. L.
  Verdier. Avec la collaboration de N. Bourbaki, P. Deligne et B. Saint-Donat.

\bibitem[Sor]{SorgerLecturesBundles}
Christoph Sorger.
\newblock Lectures on moduli of principal {$G$}-bundles over algebraic curves.
\newblock In {\em School on {A}lgebraic {G}eometry ({T}rieste, 1999)}, volume~1
  of {\em ICTP Lect. Notes}, pages 1--57. Abdus Salam Int. Cent. Theoret.
  Phys., Trieste, 2000.

\bibitem[Sus]{SuslinOnSerre}
Andrei~A. Suslin.
\newblock Projective modules over polynomial rings are free.
\newblock {\em Dokl. Akad. Nauk SSSR}, 229(5):1063--1066, 1976.

\end{thebibliography}

\end{document}